\newcommand{\Dima}[1]{{#1}}

\documentclass[10pt,a4paper]{amsart}%
\usepackage{amsfonts,amsmath,amssymb}
\usepackage{hyperref}
\usepackage{amssymb,amsthm,amsxtra}
\usepackage[usenames]{color}
\usepackage{amscd}
\usepackage{amsthm}
\usepackage{amsfonts}
\usepackage{amssymb}
\usepackage{amsmath}
\usepackage{graphicx}
\usepackage{enumerate}%
\providecommand{\U}[1]{\protect\rule{.1in}{.1in}}
\newtheorem*{theorem*}{Theorem}
\newtheorem*{lemma*}{Lemma}

\newtheorem{lemma}[subsubsection]{Lemma}
\newtheorem{proposition}[subsubsection]{Proposition}
\newtheorem{remark}[subsubsection]{Remark}

\newtheorem{theorem}[subsubsection]{Theorem}
\newtheorem{definition}[subsubsection]{Definition}

\newtheorem{conjecture}{Conjecture}
\newtheorem*{conjecture*}{Conjecture}
\newtheorem*{remark*}{Remark}
\newtheorem{thm}[subsubsection]{Theorem}
\newtheorem{prop}[subsubsection]{Proposition}
\newtheorem{lem}[subsubsection]{Lemma}
\newtheorem{defn}[subsubsection]{Definition}

\newtheorem{cor}[subsubsection]{Corollary}

\newtheorem{rem}[subsubsection]{Remark}

\newtheorem{introtheorem}{Theorem}

\baselineskip 18pt \textwidth 16cm  \theoremstyle{plain}

\newcommand{\R}{{\mathbb R}}

\newcommand{\C}{{\mathbb C}}

\newcommand{\proofend}{\hfill$\Box$\smallskip}

\newcommand{\rk}{{\operatorname{rk}}}

\newcommand{\alp}{{\alpha}}

\newcommand{\Fre}{{Fr\'{e}chet }}

\newcommand{\g}{{\mathfrak{g}}}

\newcommand{\fk}{{\mathfrak{k}}}
\newcommand{\fs}{{\mathfrak{s}}}

\newcommand{\cO}{{\mathcal{O}}}
\newcommand{\GL}{\operatorname{GL}}

\newcommand{\gl}{{\mathfrak{gl}}}

\newcommand{\n}{\mathfrak{n}}

\newcommand{\fp}{\mathfrak{p}}

\newcommand{\cV}{\mathcal{V}}

\begin{document}
\author{Dmitry Gourevitch}
\address{Dmitry Gourevitch, Faculty of Mathematics and Computer Science, Weizmann
Institute of Science, POB 26, Rehovot 76100, Israel }
\email{dimagur@weizmann.ac.il}
\urladdr{\url{http://www.wisdom.weizmann.ac.il/~dimagur}}
\author{Siddhartha Sahi}
\address{Siddhartha Sahi, Department of Mathematics, Rutgers University, Hill Center -
Busch Campus, 110 Frelinghuysen Road Piscataway, NJ 08854-8019, USA}
\email{sahi@math.rugers.edu}
\date{\today}
\title[Annihilator varieties, adduced representations, Whittaker functionals, and rank for
$\widehat{\mathrm{GL}(n)}$]{Annihilator varieties, adduced representations, Whittaker functionals, and rank for unitary
representations of $\mathrm{GL}(n)$}

\begin{abstract}
In this paper we study irreducible unitary representations of $GL_{n}\left(
\mathbb{R}\right)  $ and prove a number of results. Our first result
establishes a precise connection between the annihilator of a
representation and the existence of degenerate Whittaker functionals, for both
smooth and K-finite vectors, thereby generalizing results of Kostant, Matumoto
and others.

Our second result relates the annihilator
to the sequence of adduced representations, as defined in this setting by one of
the authors.
Based on those results, we suggest a new notion of rank of a smooth admissible
representation of $GL_{n}(\R)$, which for unitarizable representations refines
Howe's notion of rank.

Our third result computes the adduced representations for (almost)
all irreducible unitary representations in terms of the Vogan
classification.

We
also indicate briefly the analogous results over complex and $p$-adic fields.
\end{abstract}
\maketitle
\setcounter{tocdepth}{1}
\tableofcontents




\section{Introduction}

Many important problems in harmonic analysis require one to decompose a
unitary representation of a real reductive group $G$ on a
Hilbert space $\mathcal{H}$, e.g. $\mathcal{H}=L^{2}\left(  X\right)  $ for
some $G$-space $X$ equipped with an invariant measure. In order to solve such problems
one would like to know the irreducible unitary representations of $G$ as
explicitly as possible. The starting point is of course the determination of
the unitary dual $\widehat{G}$, but then it is helpful to have additional
knowledge about invariants of unitary representations, such as their
annihilator varieties, existence of Whittaker functionals, etc.

The unitary dual of $G_{n}=GL\left(  n,\mathbb{R}\right)  $ has been
determined by Vogan \cite{Vog-class} and in this paper we consider the
following invariants of $\pi\in\widehat{G_{n}}$, whose precise definitions are
given below in section \ref{invariants}:

\begin{enumerate}
\item the annihilator variety
$\mathcal{V}\left(  \pi\right)  \subset\mathfrak{gl}\left(  n,\mathbb{C}%
\right)^{\ast}$, 

\item the space $Wh_{\alpha
}^{\ast}\left(  \pi\right)  $ of (degenerate) Whittaker functionals of type
$\alpha$, 


\item the depth composition $DC\left(  \pi\right)  $ for iterated adduced representations of $\pi$, 

\item the Howe rank $HR\left(  \pi\right)  $ of $\pi$. 
\end{enumerate}

Our main theorem generalizes a number of existing results in the literature.

\begin{introtheorem}
\label{thm:Main}
Let $\pi\in\widehat{G_{n}}$ and let $\lambda$ be the partition of $n$ such that $\cV(\pi) = \overline{\cO_{\lambda}}$. Then

\begin{enumerate}

\item $Wh_{\lambda
}^{\ast}\left(  \pi\right)   \neq 0$.

\item $DC(\pi) = \lambda$.
\end{enumerate}
\end{introtheorem}

In particular, this implies that the depth composition is non-increasing.
By Matumoto's theorem (see Corollary \ref{cor:Mat} below) our result implies also that $\lambda$ is the biggest partition with $Wh_{\lambda
}^{\ast}\left(  \pi\right)   \neq 0$. By \cite[Theorems 0.1, 0.2, 4.3]{He}, $\lambda$ connects to Howe's notion of rank by $$HR\left(  \pi\right)  =\min\left(
\left\lfloor n/2\right\rfloor ,n-length(\lambda)\right) .$$ We will give an independent proof of this result (see Remark \ref{rem:SmallRep}).

Before giving the precise definitions of our invariants, we need to fix some
notation regarding partitions, compositions, nilpotent orbits, and parabolic subgroups.

\subsection{Notation} \label{subsec:not}


\begin{definition}
\label{part-comp}A \emph{composition} of $n$ of \emph{length }$k$
is a sequence ${\alpha}=({\alpha}_{1},...,{\alpha}_{k})$ of
natural numbers (i.e. strictly positive integers) such that
$\Sigma{\alpha}_{i}=n$; a \emph{partition }is a nondecreasing
composition. For a composition $\alpha$ we denote by
$\alpha^{\geq}/\alpha^{\leq}$ the nondecreasing/nonincreasing
reordering of
$\alpha$.
\end{definition}

\begin{remark}
The Young diagram of a partition is a left alligned array of boxes with
$\lambda_{i}$ boxes in row $i$. The rows of the diagram for $\lambda^{t}$ are
the columns of the diagram for $\lambda$.
\end{remark}

We sometimes use \textquotedblleft exponential\textquotedblright\ notation for
partitions; thus $4^{2}2^11^{3}$ denotes $\left(  4,4,2,1,1,1\right)  $.

We denote $\mathfrak{g}_{n}:=\mathfrak{gl}_{n}(\C)$.

\begin{definition}
\label{Jordan} If $\alpha$ is a composition of $n$, we define $J_{\alpha}%
\in\mathfrak{g}_{n}$ to be the Jordan matrix with diagonal Jordan blocks of
size ${\alpha}_{1},...,{\alpha}_{k}$; explicitly
\[
\left(  J_{\alpha}\right)  _{ij}=\left\{
\begin{array}
[c]{cc}%
1 & \text{if }j=i+1\text{ and }i\neq{\alpha}_{1}+...+\text{ }{\alpha}%
_{l}\text{ for any }l\\
0 & \text{else}%
\end{array}
\right.
\]
We define $\mathcal{O}_{\alpha}$ to be the orbit of $J_{\alpha}$ under the
adjoint action of $GL\left(  n,\mathbb{C}\right)  $ on $\mathfrak{g}_{n}$.
\end{definition}

If $\lambda=\alpha^{\geq}$ then we have $\mathcal{O}_{\alpha}=\mathcal{O}%
_{\lambda}$. Moreover by the theorem of Jordan canonical form, for each
nilpotent matrix $X$ in $\mathfrak{g}_{n}$ there is a unique partition
$\lambda$ such that $X\in\mathcal{O}_{\lambda}$.

Let $\lambda,\mu$ be partitions of $n$ and let $\overline{\mathcal{O}%
_{\lambda}}$ denote the Zariski closure of $\mathcal{O}_{\lambda}$ then we have%
\[
\mathcal{O}_{\mu}\subseteq\overline{\mathcal{O}_{\lambda}}\text{ if }%
\mu_{1}+\cdots\mu_{k} \leq \lambda_{1}+\cdots+\lambda_{k}\text{ for all
}k\text{;}%
\]
If $\lambda,\mu$ satisfy this condition, will simply write $\mu\subseteq
\lambda$.

\begin{remark}
\label{traceform} For $X,Y$ in $\mathfrak{g}_{n}$ the trace pairing is defined
to be%
\begin{equation}
\left\langle X,Y\right\rangle =trace\left(  XY\right)  \label{trace}%
\end{equation}
Then $\left\langle \cdot,\cdot\right\rangle $ is a nondegenerate symmetric
$GL\left(  n,\mathbb{C}\right)  $-invariant bilinear form (trace form) on
$\mathfrak{g}_{n}$. The trace form gives rise to an isomorphism $\mathfrak{g}%
_{n}\approx\mathfrak{g}_{n}^{\ast}$ that intertwines the adjoint and coadjoint
actions of $GL\left(  n,\mathbb{C}\right)  $. This allows us to
identify adjoint orbits and coadjoint orbits.
\end{remark}


We next fix our conventions regarding parabolic subgroups of $G_{n}$.

If $\alpha=({\alpha}_{1},...,{\alpha}_{k})$ is a composition of $n$ then we
define%
\[
S_{i}=S_{i}\left(  \alpha\right) :=\left\{  {\alpha}_{1}+\cdots+{\alpha
}_{i-1}+j\colon 1\leq j\leq\alpha_{i}\right\}
\]
For $g\in G_{n}$ we let $g_{\alpha}^{ij}$ denote the $\alpha_{i}\times
\alpha_{j}$ submatrix of $g$ with rows from $S_{i}$ and columns from $S_{j}$.

\begin{definition}
For a composition $\alpha$ we define subgroups $P_{\alpha},L_{\alpha
},N_{\alpha}$ of $G_{n}$ as follows%
\[
P_{\alpha}=\left\{  g\mid g_{\alpha}^{ij}=0\text{ if }i<j\right\}  ,L_{\alpha
}=\left\{  g\mid g_{\alpha}^{ij}=0\text{ if }i\neq j\right\}  ,N_{\alpha
}=\left\{  g\mid g_{\alpha}^{ij}=\delta^{ij}\text{ if }i\leq j\right\}  .
\]

Here $\delta^{ij}$ is Kronecker's $\delta$, while $0$ and $1$ denote zero and
identity matrices of appropriate size.
\end{definition}

Thus $B=P_{1^{n}}$ is the standard Borel subgroup of upper triangular
matrices, and each $P_{\alp}$ is a standard parabolic subgroup containing $B.$
$N_{\alpha}$ is the nilradical of $P_{\alp}$ and $P_{\alp}=L_{\alpha}N_{\alpha}$ is
a Levi decomposition with $L_{\alpha}\approx G_{{\alpha}_{1}}\times
\cdots\times G_{{\alpha}_{k}}.$

We now introduce the Bernstein-Zelevinsky product notation for parabolic induction.

\begin{definition}\label{def:BZProd}
If $\alpha=({\alpha}_{1},...,{\alpha}_{k})$ is a composition of $n$ and
$\pi_{i}\in\widehat{G_{a_{i}}}$ then $\pi_{1}\otimes\cdots\otimes\pi_{k}$ is
an irreducible unitary representation of $L_{\alpha}\approx G_{{\alpha}_{1}%
}\times\cdots\times G_{{\alpha}_{k}}$. We extend this to $P_{\alpha}$
trivially on $N_{\alpha}$ and define%
\[
\pi_{1}\times\cdots\times\pi_{k}=Ind_{P_{\alpha}}^{G_{n}}\left(  \pi
_{1}\otimes\cdots\otimes\pi_{k}\right),
\]
where $Ind$ denotes normalized induction  (see \S\S \ref{subsec:UnInd}).
\end{definition}

\begin{remark}
It follows from \cite{Vog-class} or from
(\cite{Sahi-Kirillov} and \cite{Bar}) that if $\pi_{i}\in\widehat{G_{a_{i}}}$
then $\pi_{1}\times\cdots\times\pi_{k}\in\widehat{G_{n}}$. In this case $\pi_{1}\times\cdots\times\pi_{k}$ is unchanged
under permutation of the $\pi_{i}.$
\end{remark}

\begin{remark}
Since $G_{n}/P_{\alpha}$ is compact, one can define $\pi_{1}\times\cdots
\times\pi_{k}$ analogously in the $C^{\infty}$ category.
We refer the reader to \cite{Vog-book} for details.
We will occasionally need
to consider this case especially in connection with complementary series
construction in the next section and elsewhere.
\end{remark}

\subsection{Invariants of unitary representations\label{invariants}}

For a representation $\pi$ of a Lie group $G$ in a Hilbert space, we denote by $\pi^{\infty}$ the space of smooth vectors and by $\pi^{\omega}$ the space of analytic vectors (see \S\S \ref{subsec:Analvec}). If $G$ is a real reductive group with maximal compact subgroup $K$ we also consider the Harish-Chandra module $\pi^{HC}$ consisting of $K$-finite vectors.

By \cite[Theorem 3.4.12]{Wal1} if $\pi$ is an irreducible unitary representation of a reductive group $G$ with Lie algebra $\mathfrak{g}$ and maximal compact subgroup $K$, then $\pi^{HC}$ is an irreducible $(\g,K)$-module and thus $\pi^{\infty}$ is a topologically irreducible representation of $G$.

\subsubsection{The annihilator variety and associated partition} \label{subsubsec:AnnVar}

For an associative algebra $A$ the annihilator of a module $\left(
\sigma,V\right)  $ is
\[
Ann(\sigma)=\left\{  a\in A \colon \sigma\left(  a\right)  v=0\text{ for all }v\in
V\right\}
\]
If $A$ is abelian then we define the annihilator variety of $\sigma$ to be the variety corresponding to the ideal $Ann(\sigma)$, i.e.
$\mathcal{V}\left(  \sigma\right)  =\text{Zeroes}\left(  Ann(\sigma)\right)  $.

If $\left(  \sigma,V\right)  $ is a module for a Lie algebra $\mathfrak{g}$,
then one can apply the above considerations to the enveloping algebra
$U\left(  \mathfrak{g}\right)  $. While $U\left(  \mathfrak{g}\right)  $ is
not abelian it admits a natural filtration such that $gr\left(  U\left(
\mathfrak{g}\right)  \right)  $ is the symmetric algebra $S\left(
\mathfrak{g}\right)  ,$ and one has a ``symbol" map $gr$ from $U\left(
\mathfrak{g}\right)  $ to $S\left(  \mathfrak{g}\right)  $. We let $gr\left(
Ann(\sigma)\right)  $ be the ideal in $S\left(  \mathfrak{g}\right)  $
generated by the symbols $\left\{  gr\left(  a\right)  \mid a\in Ann\left(
\sigma\right)  \right\}  $ and define the annihilator variety of $\sigma$ to
be
\[
\mathcal{V}\left(  \sigma\right)  =\text{Zeroes}\left(  gr\left(  Ann(\sigma)\right)
\right)  \subset\mathfrak{g}^{\ast}%
\]

If $\mathfrak{g}$ is a complex reductive Lie algebra and $M$ is an irreducible $\mathfrak{g}$-module, then it was shown
by Joseph (see \cite{Jos85}) that $\mathcal{V}(M)$ is
the closure $\overline{{\mathcal{O}}}$ of a single nilpotent coadjoint orbit
${\mathcal{O}}$.

If $\pi$ is a Hilbert space representation of a Lie group $G$ then we define $\cV(\pi):=\cV(\pi^{\omega})$.
If $G$ is reductive and $\pi$ is an admissible  (e.g. irreducible unitary) representation then $\pi^{HC} \subset \pi^{\omega} \subset \pi^{\infty}$, $\pi^{HC}$ is dense in $\pi^{\infty}$ and the action of $U(\g)$ is continuous. Thus, $\cV(\pi) = \cV(\pi^{HC})$.

If $\pi$ is an irreducible unitary representation then $\pi^{HC}$ is an irreducible $(\g,K)$-module and thus is a finite direct sum of algebraically irreducible representations of $\g$. These representations are $K$-conjugate and thus have the same annihilator variety. Thus \cite{Jos85} implies that $\cV(\pi)$ consists of a single nilpotent coadjoint orbit, that we call the \emph{associated orbit}.

\begin{defn}\label{def:AP}
If $\lambda$ is a partition of $n$ such that $\cV(\pi)=\overline{\cO_{\lambda}}$ we call $\lambda$ \emph{associated partition} and denote $\lambda=AP(\pi)$.
\end{defn}

For example, if $\pi$ is finite-dimensional then $\cV(\pi)=\{0\}$ and $AP(\pi)= 1^n$ and if $\pi$ is generic then, by a result of Kostant (see \S\S \ref{subsec:PrevRes}), $\cV(\pi)$ is the nilpotent cone of $\g_n^*$ and $AP(\pi)=n^1$.

For an admissible representation $\pi$ of a real reductive group
$G$ with Lie algebra $\g_{\R}$ and complexified Lie algebra $\g$ one can define more refined invariants such as

\begin{enumerate}
\item the asymptotic
support $AS\left(  \pi\right)  \subset\mathfrak{g}_{\R}^{\ast}$
\item  the wave
front set $WF\left(  \pi\right)  \subset\mathfrak{g}_{\R}^{\ast}$
(see e.g. \cite{SV}
\item the
associated variety $AV\left(  \pi\right)  \subset{\mathfrak{k}}^{\bot}
\subset\mathfrak{g}^{\ast}$ (see e.g. \cite{BV}). Here $\fk$ denotes the complexified Lie algebra of the maximal compact subgroup.
\end{enumerate}
 By \cite{Ros} and \cite{SV},
these three invariants determine each other and each of them determines
$\mathcal{V}\left(  \pi\right)  $. For $GL\left(  n,\mathbb{R}\right)  $ the
converse is true as well, and since we are primarily interested in this case
we will mostly ignore the refined invariants in this paper.

\subsubsection{Degenerate Whittaker functionals}

\label{subsec:genWhitFunc}

In this section we fix $n$, and write $N$ for $N_{\left(  1^{n}\right)  };$
thus $N$ is the subgroup of $G_{n}$ consisting of unipotent upper triangular
matrices. Let $\mathfrak{n}_{\R}$ be the Lie algebra of $N$ and let
$\mathfrak{n}$ be the complexification of $\mathfrak{n}_{\R}$.

Let $\Psi$ denote the set of multiplicative unitary characters of $N$. Then
$\Psi$ can be identified with a subset of $\mathfrak{n}^{\ast}$ via the
exponential map. More precisely we have%
\[
\Psi\approx\left\{  \psi\in\mathfrak{n}^{\ast}\mid\psi\left(  \left[
\mathfrak{n,n}\right]  \right)  =0\text{, }\psi\left(  \mathfrak{n}
_{\R}\right)  \subset i\mathbb{R}\right\}
\]
where an element $\psi$ of the right side is regarded as character of $N$ via
the formula%
\[
\psi\left(  \exp X\right)  =e^{\psi\left(  X\right)  }\text{ for }%
X\in\mathfrak{n}_{\R}.
\]
We will write $\mathbb{C}_{\psi}$ for the one-dimensional space regarded as a
module for $N$ or $\mathfrak{n}$ via $\psi$.
\begin{definition}
If $\pi\in\widehat{G_{n}}$ and $\psi\in\Psi$ we define
\[
Wh_{\psi}^{\prime}\left(  \pi\right)  =\mathrm{Hom}_{\mathfrak{n}}\left(
\pi^{HC},\mathbb{C}_{\psi}\right)  \text{, }Wh_{\psi}^{\ast}\left(
\pi\right)  =\operatorname{Hom}_{N}^{\text{cont}}(\pi^{\infty},\mathbb{C}%
_{\psi})
\]
where $\operatorname{Hom}_{N}^{\text{cont}}$ denotes the space of
\emph{continuous} $N$-homomorphisms.
\end{definition}

It is well known that $\pi^{HC}$ is dense in $\pi^{\infty}$, hence by
restriction we get an inclusion%
\begin{equation}
Wh_{\psi}^{\ast}\left(  \pi\right)  \subseteq Wh_{\psi}^{\prime}\left(
\pi\right)  \label{smhc}%
\end{equation}
Moreover $\pi^{HC}$ is finitely generated as an $\mathfrak{n}$-module and
hence we get
\[
\dim Wh_{\psi}^{\ast}\left(  \pi\right)  \leq\dim Wh_{\psi}^{\prime}\left(
\pi\right)  <\infty
\]

We refer to elements of $Wh_{\psi}^{\ast}\left(  \pi\right)  $ and $Wh_{\psi
}^{\prime}\left(  \pi\right)  $ as (degenerate) Whittaker functionals of type
$\psi$.

\begin{remark*}
Let $\mathfrak{\bar{n}}$ be the space of strictly lower triangular matrices.
Then the trace form  restricts to a nondegenerate pairing of
$\mathfrak{\bar{n}}$ with $\mathfrak{n}$, allowing one to identify
$\mathfrak{\bar{n}\approx n}^{\ast}$. Under this identification elements of
$\Psi$ correspond to imaginary ``subdiagonal" matrices, i.e \ to matrices
$X\in\mathfrak{g}_{n}$ satisfying%
\[
X_{pq}\in i\mathbb{R}\text{ if }p=q+1\text{ and }X_{pq}=0\text{ if }p\neq q+1
\]
\end{remark*}

Since we have identified $\mathfrak{g}_{n}^{\ast}\approx\mathfrak{g}_{n}$ via
the trace form, the above remark also allows us to regard $\mathfrak{n}^{\ast
}$ as a subspace of $\mathfrak{g}_{n}^{\ast}$. Hence we may also regard $\Psi$
as a subspace of $\mathfrak{g}_{n}^{\ast}$.

\begin{defn}
Let $\alpha$ be a composition of $n$, $J_{\alpha}$ be the corresponding Jordan matrix as in Definition \ref{Jordan}, and $w$ be the longest element of the Weyl group. Then $wJ_{\alpha}w^{-1}$ is subdiagonal and $iwJ_{\alpha}w^{-1}$ can be
regarded as an element of $\Psi$ by the above remark. We denote this character by $\psi_{\alpha}$.
Note that $\psi_{\alpha}\in{\mathcal{O}}_{\alpha}$.

For $\pi\in\widehat{G_{n}}$ denote also
\begin{equation}
Wh_{\alpha}^{\prime}\left(  \pi\right)  = Wh_{\psi_{\alpha}}^{\prime}\left(  \pi\right),  \quad Wh_{\alpha}^{\ast}\left(
\pi\right)  =  Wh_{\psi_{\alpha}}^{\ast}\left(
\pi\right)
\end{equation}
\end{defn}
%
\subsubsection{The adduced representation, \textquotedblleft derivatives \textquotedblright and the depth composition}

\label{subsec:DerDepth}

In \cite{BZ-Induced,Ber} Bernstein and Zelevinsky introduced the important notion
of \textquotedblleft derivative\textquotedblright\ for representations of
$GL\left(  n,\mathbb{Q}_{p}\right)  $. An Archimedean analog of the
\textquotedblleft highest\textquotedblright\ derivative for $\pi
\in\widehat{G_{n}}$ was defined in \cite{Sahi-Kirillov}, where it was called
the adduced representation and denoted $A\pi$. This definition which we now
recall, involves two ingredients.

Let $P_{n}\subset G_{n}$ be the \textquotedblleft mirabolic\textquotedblright%
\ subgroup consisting of matrices with last row $(0,0,...,0,1)$. To forestall
confusion we note that $P_{n}$ is \emph{not} a parabolic subgroup, it has
codimension $1$ in $P_{\left(  n-1,1\right)  }$ \ and is completely different
from $P_{({n})}=G_{n}$.

The first ingredient in the definition of $A\pi$ is the following result that
was conjectured by Kirillov.

\begin{thm}
\label{BarKir} Let $\pi\in\widehat{G_{n}}$, then $\pi|_{P_{n}}$ is irreducible.
\end{thm}

\noindent This was first proven in the $p$-adic case in \cite{Ber}, then in the
complex case in \cite{Sahi-Kirillov}, and finally in the real case in
\cite{Bar}. Three new proofs have been obtained recently in \cite{AG,SZ,GaLa}.

The second ingredient is Mackey theory that describes the unitary dual of Lie
groups, such as $P_{n}$, which are of the form $G=H\ltimes Z$ \ with $Z$
abelian. In this case $\widehat{Z}$ consists of unitary characters and $H$
acts on $\widehat{Z}$; for $\chi\in\widehat{Z}$ \ let $S_{\chi}$ denote its
stabilizer in $H$. If $\sigma\in\widehat{S_{\chi}}$ then $\sigma\otimes\chi$
is a unitary representation of $S_{\chi}\ltimes Z$, and we define
\[
I_{\chi}\left(  \sigma\right)  =Ind_{S_{\chi}\ltimes Z}^{G}\left(
\sigma\otimes\chi\right)
\]
The main results of Mackey theory are as follows:
\begin{enumerate}[(a)]
\item $I_{\chi}\left(
\sigma\right)  $ is irreducible for all $\sigma\in\widehat{S_{\chi}}$;

\item
$\widehat{G}$ is the disjoint union of $I_{\chi}\left(  \widehat{S_{\chi}%
}\right)  $ as $\chi$ ranges over representatives of distinct $H$-conjugacy
classes in $\widehat{Z}$.
\end{enumerate}

Since $P_{n}\approx G_{n-1}\ltimes{\mathbb{R}}^{n-1}$ we may analyze
$\widehat{P_{n}}$ by Mackey theory. There are two $G_{n-1}$-conjugacy classes
in $\widehat{\R^{n-1}}$; one class consists of the trivial character $\chi_{0}$
alone, while the other class contains all other characters. As a
representative of the second class we pick the character $\chi_{1}$ defined by%
\[
\chi_{1}(a_{1},...,a_{n-1})=\exp\left(  ia_{n-1}\right)
\]
The stabilizers in $G_{n-1}$ are $S_{\chi_{0}}=G_{n-1}$ and $S_{\chi_{1}%
}=P_{n-1}$ and therefore we get%

\begin{equation}
\widehat{P_{n}}=I_{\chi_{0}}\left(  \widehat{G_{n-1}}\right)  \coprod
I_{\chi_{1}}\left(  \widehat{P_{n-1}}\right)  \label{PGP}%
\end{equation}
We may iterate (\ref{PGP}) until we arrive at the trivial group $P_{1}=G_{0}$
\[
\widehat{P_{n}}=\coprod_{k=1}^{n}I_{\chi_{1}}^{k-1}I_{\chi_{0}}\left(
\widehat{G_{n-k}}\right)
\]
Combining this with Theorem \ref{BarKir} we deduce that for $\pi
\in\widehat{G_{n}}$ with $n>0$, there exists a unique natural number $d$ and a
unique $\pi^{\prime}\in\widehat{G_{n-d}}$ such that
\begin{equation}
\pi|_{P_{n}}=I_{\chi_{1}}^{d-1}I_{\chi_{0}}\left(  \pi^{\prime}\right)
\label{der}%
\end{equation}

\begin{definition}
\label{Api} \cite{Sahi-Kirillov} If $\pi\in\widehat{G_{n}}$ and $\pi^{\prime
}\in\widehat{G_{n-d}}$ satisfy (\ref{der}) we say that $\pi$ has depth $d$ and
that $\pi^{\prime}$ is the adduced representation (or highest derivative) of
$\pi$, and we write $\pi^{\prime}=A\pi$.
\end{definition}

Let $\star$ denote the trivial representation of the trivial group $G_{0}$. The
procedure of taking the adduced representation can be iterated until we arrive at
$\star$. Thus we obtain a sequence of unitary representations
\begin{equation}
\left(  \pi_{0}=\pi,\pi_{1},\ldots,\pi_{l-1},\pi_{l}=\star\right)  \text{
satisfying }\pi_{j}=A\left(  \pi_{j-1}\right)  \text{ for }1\leq j\leq
l\text{.} \label{derseq}%
\end{equation}
and we write $d_{i}=$ $d_{i}\left(  \pi\right)  $ for the depth of $\pi_{i-1}$. Note that $d_{1},d_{2},\ldots,d_{l}$ are natural numbers, and their sum is
precisely $n$.

\begin{definition}
\label{DP} The composition $\left(  d_{1},d_{2},\ldots,d_{l}\right)  $ is called
the \emph{depth composition} of $\pi\in\widehat{G_{n}}$ and is denoted
$DC\left(  \pi\right)  $.
\end{definition}

\subsubsection{Howe Rank}\label{subsec:HowR}

$P_{\alpha}$ is called a maximal parabolic if $\alpha$ has length $2$ so that
$\alpha=\left(  a,b\right)  $ with $a+b=n$. In this case $N_{\alpha}$ is
abelian and isomorphic to $M_{a\times b}$, the additive group of $a\times b$
real matrices. The unitary dual of $N_{\alpha}$ consists of unitary
characters and can also be identified with $M_{a\times b}$ via $\chi
_{y}\left(  x\right)  =\exp\left(  iTr\left(  xy^{t}\right)  \right)  $. The
group $L_{\alpha}=G_{a}\times G_{b}$ acts on $M_{a\times b}$ in the usual
manner, and the orbits are $\mathcal{R}_{0},\ldots,\mathcal{R}_{\min\left(
a,b\right)  }$ where $\mathcal{R}_{k}$ denotes the set of matrices of rank
$k$. Note that $\min\left(  a,b\right)  \leq\left\lfloor n/2\right\rfloor $
and that equality holds for $a=\left\lfloor n/2\right\rfloor ,b=n-\left\lfloor
n/2\right\rfloor $.

We now briefly describe the theory of Howe rank for $G_{n}$; let us fix
$\alpha=\left(  a,n-a\right)  $ as above. If $\pi\in\widehat{G_{n}}$\ then by
Stone's theorem the restriction $\pi|_{N_{\alpha}}$ corresponds to a
projection-valued Borel measure $\mu^{\pi}$ on $\widehat{N_{\alpha}}\approx
M_{a\times\left(  n-a\right)  }$. Since $\pi$ is a representation of
$P_{\alpha}$, $\mu^{\pi}$ is $P_{\alpha}$-invariant and decomposes as a direct
sum
\[
\mu^{\pi}=\mu_{0}^{\pi}+\cdots+\mu_{\min\left(  a,n-a\right)  }^{\pi}\text{
with }\mu_{k}^{\pi}\left(  E\right)  =\mu^{\pi}\left(  E\cap\mathcal{R}%
_{k}\right)
\]
Building on the work of Howe \cite{How}, Scaramuzzi \cite{Sca} proved that
$\mu_{\pi}$ has \textquotedblleft pure rank\textquotedblright, i.e. there is
some integer $k=HR\left(  \pi,a\right)  $ such that $\mu^{\pi}=\mu_{k}^{\pi}$.
Moreover if we define $HR\left(  \pi\right)  =HR\left(  \pi,\left\lfloor
n/2\right\rfloor \right)  $ then
\[
HR\left(  \pi,a\right)  =\min\left(  HR\left(  \pi\right)  ,a,n-a\right)
\text{ for all }a\leq n
\]

\begin{definition}
\label{HR} For $\pi\in\widehat{G_{n}}$ the integer $HR\left(  \pi\right)  $ is
called the Howe rank of $\pi.$
\end{definition}


\subsection{Results over other local fields and a uniform formulation of
Theorem \ref{thm:Main}}

Theorem \ref{thm:Main} also holds in the complex case, and the proof is very
similar. We comment on that in section \ref{sec:Complex}.


In section \ref{sec:padic_short} we prove a $p$-adic analog of Theorem
\ref{thm:Main}, using \cite{BZ-Induced,MW,Zl}. In the $p$-adic case one cannot define annihilator variety, but one can consider the wave front set $WF(\pi)$.
It is not known to be the closure of a single nilpotent orbit for general reductive group but for irreducible smooth representations of $GL_n(F)$ this is proven to be the case in \cite{MW}.

Since in the $p$-adic case the
notion of derivative is defined for all smooth representations, this analog
does not require the representation to be unitary. This gives us the following
uniform formulation of the main theorem.

\begin{introtheorem}
\label{thm:UniformMain} Let $F$ be a local field of characteristic zero. Let
$\pi$ be an irreducible smooth admissible representation of $\operatorname{GL}%
(n,F))$ and let $WF(\pi) \subset \gl_n(F)$ denote the wave front set of $\pi$. Let $\lambda$ be the partition of $n$ such that $WF(\pi) = \overline{{\mathcal{O}}_{\lambda}}$.
Suppose that either $F$ is non-Archimedean or $\pi$ is unitarizable. Then

\begin{enumerate}
\item \label{it:BWh} $Wh_{\lambda}^{\ast}(\pi)\neq0$, and for any composition ${\alpha}$ with
$Wh_{{\alpha}}^{\ast}(\pi)\neq0$ we have ${\mathcal{O}}_{{\alpha}}%
\subset\overline{{\mathcal{O}}_{\lambda}}$.

\item \label{it:BDC} $DC\left(  \pi\right)  =\lambda$; in particular $DC\left(
\pi\right)  $ is a non-increasing sequence.

\item \label{it:BCP} $\lambda$ is the transpose of the classification partition of $\pi$ (see Remark \ref{rem:CP} below).
\end{enumerate}

Moreover, if $\pi$ is unitarizable then $HR\left(  \pi\right)  =\min\left(
\left\lfloor n/2\right\rfloor ,n-length(\lambda)\right)  $.
\end{introtheorem}

If $F$ is non-Archimedean we also have $\dim Wh_{DC\left(  \pi\right)
}^{\ast}(\pi)=1$.

In the $p$-adic case, $Wh_{\alp}^*$ denotes the space of all linear equivariant functionals, since all representations are considered in discrete topology.

\begin{remark}\label{rem:CP}
The classification partition is defined in \cite{Ven} for all irreducible unitary representations, through the Tadic-Vogan classification, and in \cite{OS2} for all irreducible smooth representations through the Zelevinsky classification. For representations of Arthur type, this partition describes the $SL(2)$-type of the representation of the Weil-Deligne group corresponding to $\pi$ by local Langlands correspondence.
\end{remark}
In the Archimedean case, Theorem \ref{thm:UniformMain} follows from Theorem \ref{thm:Main}, Corollary \ref{cor:Mat}, Theorem \ref{thm:AVTadic} and Remark \ref{rem:SmallRep} below.

\subsection{Earlier works}\label{subsec:PrevRes}

It was proven by Casselman-Zuckerman for $G_n$ (unpublished) and by Kostant (\cite{Kos}) for all quasi-split reductive groups that for a generic character $\psi$ of $N$, $Wh^*_{\psi}(\pi) \neq 0$ if and only if $Wh'_{\psi}(\pi) \neq 0$ if and only if $\cV(\pi)$ is maximal possible, i.e. equal to the nilpotent cone.
Matumoto (\cite[Theorem 1]{Mat}) proved a generalization of one direction of this statement. For the case of $G_n$ his theorem implies

\begin{cor}
\label{cor:Mat} If $\alp$
is a composition of $n$, and $M$ is a ${\mathfrak{g}}_{n}$-module such that $\operatorname{Hom}
_{\n}(M,\psi_{\alpha})\neq0$ then $\mathcal{O}_{\alpha}\subset\mathcal{V} (M)$.
\end{cor}


Over $p$-adic fields, a connection between wave front set and generalized
Whittaker functionals was investigated in \cite{MW} for smooth (not necessary
unitarizable) representations of any reductive group. However, the main
theorem of \cite{MW} involves a lot of choices and for the case of $GL_{n}$
can be made much more concrete, using derivatives and Zelevinsky
classification, following \cite{BZ-Induced,Zl}.

Several works (e.g. \cite{GW,MatDuke,MatActa}) give a partial Archimedean analog of the results of \cite{MW}, by considering
non-degenerate (``admissible") characters of the (smaller) nilradicals of
bigger parabolic subgroups. However, much less is known in the Archimedean case. Our work establishes a different type of analog:
instead of considering non-degenerate characters of smaller nilradicals, we
consider degenerate characters of the nilradical of the standard Borel
subgroup. Following Zelevinsky \cite[\S\S 8.3]{Zl}, we call functionals equivariant with respect
to such characters \emph{degenerate Whittaker functionals}.

\subsection{Generalizations in future works}

\label{subsec:GenGroup} In \cite{AGS},
we define the notion of highest derivative for all admissible smooth
\Fre representations of $G_{n}$. We show that this notion extends the notion of adduced representation discussed in the current paper, and establish several properties of highest derivative analogous to the ones proven
in the $p$-adic case in \cite{BZ-Induced}.
We apply those properties to questions raised in the current paper. Namely, we complete the computation of adduced representations for all Speh complementary series, and prove that $\dim Wh_{DC\left(  \pi\right)
}^{\ast}(\pi)=1$ for all $\pi \in \widehat{G_n}$.

In our work in progress \cite{GS} we prove that
\[
Wh_{\alpha}^{\ast}\left(  \pi\right)  \neq0\Leftrightarrow Wh_{\alpha}%
^{\prime}\left(  \pi^{HC}\right)  \neq0\Leftrightarrow{\mathcal{O}}_{\alpha
}\subset\mathcal{V}(\pi).
\]
for any irreducible admissible smooth \Fre representation
$\pi$ of $G_{n}$ and any composition $\alpha$ of $n$.

Furthermore, we prove the following generalization for \Dima{any quasi-split real reductive
group $G$.} Let $N$ be the nilradical of a Borel subgroup of $G$, and $K$ be  maximal compact subgroup of $G$. Let
$\mathfrak{g}$ and $\mathfrak{n}$ be the complexified Lie algebras of $G$ and
$N$. Let $\pi$ be an irreducible admissible smooth representation of $G$. Denote $$\Psi(\pi):=\{\psi\in\Psi \text{ s.t. }
Wh^*_{\psi}(\pi) \neq0\}, \, \text{ and } \, \Psi(\pi^{HC}):=\{\psi\in\Psi \text{ s.t. } Wh'_{\psi}(\pi) \neq0\}.$$
In \cite{GS} we prove
\begin{enumerate}
\item $\Psi(\pi^{HC})=pr_{\mathfrak{n}^{*}}(AV(\pi^{HC})) \Dima{\cap \Psi}$
\label{Cit:HC}

\item $\Psi(\pi) \subset  \Dima{\Psi(\pi^{HC}) \cap i\Psi({\mathbb{R}})}$, and under certain condition on $G$ equality holds.
\label{Cit:real}
\end{enumerate}

%

\subsection{The proposed notion of rank}

We suggest the following definition of rank.

\begin{defn} \label{def:rank}
If $\pi$ is a 
smooth \Fre representation of $G_{n}$, we
define the rank of $\pi$, written $\rk(\pi)$, to be the maximum rank of a matrix in ${
\mathcal{V}(\pi)}$.
\end{defn}

If $\mathcal{V}(\pi)$ is a closure of a single orbit given by a partition of
length $k$ then $\rk(\pi)=n-k$.
Theorem \ref{thm:UniformMain} implies that for $\pi\in\widehat{G_{n}}$ our notion of
rank agrees with the notion proposed in \cite{Sahi-Kirillov}, and extends
Howe's notion of rank. In subsection \ref{subsec:CompAV} we compute (extended) ranks
of all irreducible unitary representations of $G_{n}$ in terms of the
Vogan classification.

Our definition extends literally to all classical groups, and connects to
Howe's notion of rank by \cite[Theorem 0.4]{He}.

\Dima{For $G_n$ we can} give another interpretation of rank, in terms of parabolic induction.
Let $\lambda$ be a partition of $n$, $\alpha:=\lambda^{\leq}$ be the inverse reordering of $\lambda$, $P_{\alp}$ the corresponding standard
parabolic subgroup, $L_{\alp}$ its Levi component, $N_{\alp}$ its
nilradical, and $\mathfrak{n}_{\alp}$ its Lie algebra. Consider the ``naive"
(non-exact) Jacquet restriction functor $r_{\alp}$ that maps $\pi$ to
$\pi/\mathfrak{n}_{\alp}\pi$. This functor is adjoint to parabolic
induction from $P_{\alp}$. Note that $Wh_{\alp}^{\ast}(\pi)$ is equal to
the space $Wh^{\ast}(r_{\alp}(\pi))$ of generic (classical) Whittaker
functionals on $r_{\alp}(\pi)$. Therefore, Theorem \ref{thm:Main}
%
implies that for $\pi\in\widehat{G_{n}}$, $\rk(\pi)\geq k$ if and only if the
Jacquet restriction of $\pi$ to some Levi subgroup of semi-simple rank $k$ is generic.

The $\rk(\pi)$ is also equal to the (real) dimension of the variety consisting of
unitary characters $\psi$ of $N$ such that $Wh_{\psi}^{\ast}(\pi)\neq0$. 



\Dima{For $G\neq G_{n}$ we do not have equivalent descriptions of the rank using Jacquet functor or the space of functionals.}

\subsection{Application to construction of Klyachko models}

Our technique can be applied to construction of another family of models for
unitary representations of $G_{n}$, called Klyachko models. Let $F$ be a local
field.
For any decomposition $n=r+2k$ define a subgroup $H_{r,2k}\subset G_{n}$ ,
which is a semi-direct product of $Sp_{2k}(F)$ and the nilradical $N_{\left(
1,...,1,2k\right)  }$ of the standard parabolic $P_{\left(  1,...,1,2k\right)
}$, and let $\phi_{2k,r}$ denote the generic character $\psi_{(n)}$ restricted to $N_{\left(
1,...,1,2k\right)  }$ and then extended trivially to $H_{r,2k}$. Then the $r$-th
Klyachko model of $\pi\in\widehat{G_{r+2k}}$ is defined to be
$\operatorname{Hom}_{H_{r,2k}}(\pi^{\infty},\phi_{r,2k})$.

Offen and Sayag proved existence, uniqueness and disjointness of Klyachko
models over $p$-adic fields.
They defined the appropriate $r$ in terms of Tadic classification, used
derivatives to reduce the statement to the case $r=0$, in which an
$Sp$-invariant functional is constructed using a global (automorphic)
argument. Theorem \ref{thm:UniformMain} allows one to define $r$ in terms of
the partition corresponding to $\pi$, and then to extend the construction of
Klyachko models given in \cite[\S 3]{OSConst} to the Archimedean case. This is
done in \cite{GOSS}.

\subsection{Structure of the paper}

In section \ref{sec:Prel} we give several necessary definitions and
preliminary results on geometry of coadjoint orbits, analytic and smooth
vectors, and annihilator varieties.



In section \ref{sec:Der} we prove the main theorem by induction.
First we note that $\pi|_{P_n}$ is induced from $A\pi$. This gives us a map $\pi^{\infty} \to (A\pi)^{\infty}$ which has dense image and satisfies a certain equivariant condition with respect to a subgroup of $P_n$ that includes $N$. This also enables us to compute the dimension of $\cV(\pi|_{P_n})$ in terms of $\dim \cV(A\pi)$.
Using the induction hypothesis and the map $\pi^{\infty} \to (A\pi)^{\infty}$, we show that $Wh^*_{DC(\pi)}(\pi) \neq 0$. By Matumoto's theorem this implies $\overline{\cO_{DC(\pi)}} \subset \cV(\pi)$. To show equality we prove that the dimension of the annihilator variety does not drop when we restrict $\pi$ to $P_n$, and then use the induction hypothesis.

In section \ref{sec:CompDer} we compute the adduced representation for
representations of Speh complementary series (except one special case), thus
almost finishing the computation of adduced representations of all unitary
irreducible representations. We use the restriction on the annihilator variety
of the adduced representation given by Theorem \ref{thm:Main} and a restriction on
the infinitesimal character of the adduced representation that we deduce from
Casselman - Osborne lemma. Those two restrictions allow us to determine the
adduced representation uniquely, except in one case where we present two possibilities.

In section \ref{sec:Complex} we explain how our proofs adapt to the complex case.


In section \ref{sec:padic_short} we deal with the $p$-adic case. Most of the
arguments here are simply sketched since the components of various proofs
appear already in \cite{Zl} and \cite{MW}.



\subsection{Acknowledgements}

We wish to thank Avraham Aizenbud, Joseph Bernstein, Roe Goodman, Anthony
Joseph, and Nolan Wallach
for fruitful discussions and Gordan Savin \Dima{and the referee} for useful remarks.

This paper was conceived when D.G. was a Hill Assistant
Professor at Rutgers University. He wishes to thank the
Mathematics department for warm and fruitful atmosphere and good
working conditions.

\section{Preliminaries}

\label{sec:Prel}

\subsection{Induction and dimensions of nilpotent orbits}
\label{subsec:GeoOrbits}

We now recall Lusztig-Spaltenstein induction of nilpotent orbits. Let
${\mathfrak{g}}$ be a complex reductive Lie algebra. Let $\mathfrak{p}$ be a
parabolic subalgebra with nilradical $\mathfrak{n}$ and Levi quotient
$\mathfrak{l}:=\mathfrak{p}/\mathfrak{n}$. Then we have a natural projection
$pr\colon {\mathfrak{g}}^{\ast}\twoheadrightarrow\mathfrak{p}^{\ast}$ and a natural
embedding $\mathfrak{l}^{\ast}\subset\mathfrak{p}^{\ast}$.

\begin{thm}
[\cite{CoMG}, Theorem 7.1.1]In the above situation let ${\mathcal{O}%
}_{\mathfrak{l}}\subset\mathfrak{l}^{\ast}$ be a nilpotent orbit. Then there
exists a unique nilpotent orbit ${\mathcal{O}}_{{\mathfrak{g}}}$ that meets
$pr^{-1}({\mathcal{O}}_{\mathfrak{l}})$ in an open dense subset. We have
\[
\dim{\mathcal{O}}_{{\mathfrak{g}}}=\dim{\mathcal{O}}_{\mathfrak{l}}+2\dim
\mathfrak{n}.
\]

\end{thm}

\begin{defn}
The orbit ${\mathcal{O}}_{{\mathfrak{g}}}$ is denoted $Ind_{\mathfrak{l}%
}^{{\mathfrak{g}}}({\mathcal{O}}_{\mathfrak{l}})$ and called the induced orbit
of ${\mathcal{O}}_{\mathfrak{l}}$.
\end{defn}

\begin{thm}
[\cite{BB2}, Theorem 2]\label{thm:BarBoz} Let $G$ be a real reductive Lie group. If $\pi$ is an irreducible
representation of $G$ that is parabolically induced from a representation
$\sigma$ of a Levi subgroup $L$, then $\mathcal{O}\left(  \pi\right)  =$
$Ind_{\mathfrak{l}}^{{\mathfrak{g}}}\left[  \mathcal{O}\left(  \sigma\right)
\right]  $.
\end{thm}

We now specialize this discussion to ${\mathfrak{g=g}}_{n}$. In this case nilpotent obits are described by partitions.

\begin{definition}
\label{l+m} If $\lambda,\mu$ are partitions of $p,q$ of lengths $k,l$
respectively, we define a partition $\lambda+\mu$ of $p+q$ of length
$m=\max\left(  k,l\right)  $ as follows:
\[
(\lambda+\mu)_{i}=\lambda_{i}+\mu_{i}\text{ for }i\leq m
\]
where the missing $\lambda_{i},\mu_{i}$ are treated as $0$.
\end{definition}

\begin{remark}
We can describe $\lambda+\mu$ using transposed partitions: $\left(
\lambda+\mu\right)  ^{t}$ is the partition rearrangement of the composition
$\left(  \lambda^{t},\mu^{t}\right)  $. On the level of Young diagrams:
$\lambda+\mu$ is obtained by concatenating the columns of
$\lambda$ and $\mu$ and reordering them in descending order.
\end{remark}

\begin{prop}
[\cite{CoMG}, Lemma 7.2.5]\label{prop:LSInd} Let $\lambda$ be a partition of
$l$, and $\mu$ be a partition of $m$. Then
\[
Ind_{{\mathfrak{g}}_{l}\times{\mathfrak{g}}_{m}}^{{\mathfrak{g}}_{l+m}%
}({\mathcal{O}}_{\lambda}\times{\mathcal{O}}_{\mu})={\mathcal{O}}_{\lambda
+\mu}.
\]
\end{prop}


\begin{prop}[ \cite{CoMG}, Corollary 7.2.4.]
\label{prop:DimOrbT} If $\lambda$ is a partition of $n$ and $\nu=\lambda^{t}$
then
\[
\dim\mathcal{O}_{\lambda}=n^{2}-\sum\nolimits_{j}\nu_{j}^{2}.
\]

\end{prop}

We recall the BZ-product notation for parabolic induction as in Definition \ref{def:BZProd} and the associated partition $AP(\pi)$ as in Definition \ref{def:AP}. From Proposition
\ref{prop:LSInd} we obtain
\begin{cor}
\label{cor:BarBoz} If ${\sigma\times\tau}$ is irreducible then
\[
AP\left(  {\sigma\times\tau}\right)  =AP\left(  {\sigma}\right)  +AP\left(
{\tau}\right)  .
\]
\end{cor}

\subsection{Unitary induction} \label{subsec:UnInd}

Let $G$ be a Lie group and let $dx$ denote a \emph{right} invariant Haar measure  then we have%
\[
\int\delta_{G}\left(  g\right)  f\left(  gx\right)  dx=\int f\left(  x\right)
dx\text{ }%
\]
where $\delta_{G}\left(  g\right)  =\left\vert \det\text{ }Ad\left(  g\right)
\right\vert $ is the modular function of $G$.

Let $H$ be a closed subgroup of $G$ and write $\delta_{H\backslash G}\left(
h\right)  =\delta_{H}\left(  h\right)  \delta_{G}\left(  h\right)  ^{-1}$. We
define $C_{c}\left(  H\backslash G,\delta_{H\backslash G}\right)  $ to be the
space of continuous functions $f\in C\left(  G\right)  $ that satisfy%
\begin{equation}
f\left(  hx\right)  =\delta_{H\backslash G}\left(  h\right)  f\left(
x\right)  \text{ for all }h\in H\text{ } \label{transf}%
\end{equation}
and for which supp$\left(  f\right)  \subseteq$ $HS$ for some compact set $S$.
Then $G$ acts on $C_{c}\left(  H\backslash G,\delta_{H\backslash G}\right)  $ by right translations
and there is a unique continuous $G$-invariant functional $\int_{H\backslash
G}$ on $C_{c}\left(  H\backslash G,\delta_{H\backslash G}\right) $ satisfying%
\[
\int_{H\backslash G}\left[  \int_{H}\delta_{H\backslash G}\left(  h\right)
^{-1}f\left(  xh\right)  dh\right]  =\int_{G}f\left(  x\right)  dx\text{ for
all }f\in C_{c}\left(  G\right)
\]
One shows that $f\mapsto\int_{H}\delta_{H\backslash G}\left(  h\right)
^{-1}f\left(  xh\right)  dh$ is a continuous surjection from $C_{c}\left(
G\right)  $ to $C_{c}\left(  H\backslash G\right)  $ whose kernel is densely
spanned by functions $f\left(  x\right)  -\delta_{G}\left(  h\right)  f\left(
hx\right)  $. Consequently $\int_{G}dx$ vanishes on the kernel and descends to
a functional on $C_{c}\left(  H\backslash G,\delta_{H\backslash G}\right)$.

If $(\sigma,V)$ is a unitary representation of $H$, we define $C_{c}\left(
H\backslash G,\delta_{H\backslash G}^{1/2}\sigma\right)  $ to be the set of
continuous functions $f\colon G\rightarrow V$ that satisfy\ %

\begin{equation}
f\left(  hg\right)  =\delta_{H\backslash G}^{1/2}\left(  h\right)
\sigma\left(  h\right)  f\left(  g\right)  \label{eq:IndTwidle}%
\end{equation}
and for which supp$\left(  f\right)  \subseteq$ $HS$ for some compact set $S$.
For such $f$ we have $\left\Vert f\left(  x\right)  \right\Vert _{V}^{2}\in
C_{c}\left(  H\backslash G,\delta_{H\backslash G}\right)  $ and we define $W$
to be the closure of $C_{c}\left(  H\backslash G,\delta_{H\backslash G}%
^{1/2}\sigma\right)  $ with respect to the norm $\int_{H\backslash
G}\left\Vert f\left(  x\right)  \right\Vert _{V}^{2}$. Then $W$ is a Hilbert
space under $\left\langle f_{1},f_{2}\right\rangle =\int_{H\backslash
G}\left\langle f_{1}\left(  g\right)  ,f_{2}\left(  g\right)  \right\rangle
_{V}$ and the action of $G$ by right translations defines a unitary
representation $\left(  \pi,W\right)  $ of $G$.

\begin{definition}
\label{def:UnInd} In the above situation $\left(  \pi,W\right)  $ is called
the unitarily induced representation and denoted by $Ind_{H}^{G}(\sigma)$.
\end{definition}

\subsection{Analytic and smooth vectors}

\label{subsec:Analvec}

\begin{defn}
Let $M$ be an analytic manifold and $B$ be a Banach space. A map $M\rightarrow
B$ is said to be \emph{smooth} if it is infinitely differentiable, and
\emph{analytic} if its Taylor series at every point has a positive radius of convergence.
\end{defn}

\begin{rem}
A map $f \colon M\rightarrow B$ is smooth/analytic if and only if the composition
$\phi\circ f$ is smooth/analytic for every continuous linear functional
$\phi \colon B\rightarrow{\mathbb{C}}$ .
\end{rem}

\begin{defn}
Let $\left(  \sigma,B\right)  $ be a continuous Banach representation of a Lie
group $G$. A vector $v\in B$ is called smooth/analytic if the action map
$G\rightarrow B$ defined by $g\mapsto\sigma(g)v$ is smooth/analytic. Both $G$
and its Lie algebra ${\mathfrak{g}}$ act on the spaces of smooth and analytic
vectors and we denote the corresponding representations by $\left(
\sigma^{\infty},B^{\infty}\right)  $ and $\left(  \sigma^{\omega},B^{\omega
}\right)  $ respectively. $\left(  \sigma^{\infty},B^{\infty}\right)  $ is
naturally a \Fre representation of $G$.
\end{defn}

\begin{rem}
\label{rem:dense} By \cite[Theorem 4]{Nel}, $B^{\omega}$ is dense in
$B$.
\end{rem}

The following theorem by Poulsen can be interpreted as a
representation-theoretic version of Sobolev's embedding theorem.

\begin{thm}
[\cite{Pou}, Theorem 5.1 and Corollary 5.2]\label{thm:Pou} Let $\left(
\pi,W\right)  =Ind_{H}^{G}(\sigma)$ ; if $f\in W^{\infty}$ ($W^{\omega}$) then
$f$ is a smooth (analytic) function from $G$ to $V$.
\end{thm}

\begin{cor}
Let $\left(  \pi,W\right)  =Ind_{H}^{G}(\sigma)$ be as above; if $f\in
W^{\infty}$ ($W^{\omega}$) then $f$ takes values in $V^{\infty}$ ($V^{\omega}$).
\end{cor}

\begin{proof}
Let $f\in W^{\infty}$ ($W^{\omega}$) and let $v:=f(g)$, then the action map of
$H$ on $v$ is
\[
h\mapsto\sigma\left(  h\right)  v=\sigma\left(  h\right)  f\left(  g\right)
=\delta_{H\backslash G}^{-1/2}\left(  h\right)  f\left(  hg\right)
\]
This is a smooth (analytic) map by Poulsen's theorem; hence $v$ is in
$V^{\infty}$ ($V^{\omega}$).
\end{proof}

\begin{cor}
\label{cor:IndFunc} Let $\left(  \pi,W\right)  =Ind_{H}^{G}(\sigma)$; then
$f\mapsto f(e)$ defines a continuous $H$-equivariant morphism $\left(  \pi^{\infty
},W^{\infty}\right)  \twoheadrightarrow\left(  \widetilde{\sigma}^{\infty
},V^{\infty}\right)  ,$ where $\widetilde{\sigma}\left(  h\right)
=\delta_{H\backslash G}^{1/2}\left(  h\right)  \sigma\left(  h\right)  $ as in
(\ref{eq:IndTwidle}) above.
\end{cor}
The continuity of the evaluation morphism follows from \cite[Lemma 5.2]{Pou}.

\begin{definition}
Let $V_{1},V_{2}$ be two modules for a Lie algebra ${\mathfrak{g}}$; we say
$V_{1},V_{2}$ are non-degenerately ${\mathfrak{g}}$-paired if there is a
non-degenerate bilinear pairing $\left\langle \cdot , \cdot \right\rangle $ that is
invariant in the Lie algebra sense, \emph{i.e.}
\[
\left\langle Xv_{1},v_{2}\right\rangle =-\left\langle v_{1},Xv_{2}%
\right\rangle \text{ for all }v_{1}\in V_{1},v_{2}\in V_{2},X\in{\mathfrak{g}}%
\]

\end{definition}

\begin{lemma}
\label{lem:IndPairing} Let $\left(  \pi,W\right)  =Ind_{H}^{G}(\sigma)$;
suppose $G/H$ is connected and $\widetilde{\sigma}^{\omega}$is
non-degenerately ${\mathfrak{h}}$-paired with an ${\mathfrak{h}}$-module
$\tau$. Then $\pi^{\omega}$ is non-degenerately ${\mathfrak{g}}$-paired with a
quotient of $U({\mathfrak{g}})\otimes_{U({\mathfrak{h}})}\tau$.
\end{lemma}

Here and elsewhere ${\mathfrak{g}}$, ${\mathfrak{h}}$ denote the Lie algebras
of $G$, $H$ and $U({\mathfrak{g}})$, $U({\mathfrak{h}})$ their enveloping algebras.

\begin{proof}
As noted above ${\mathfrak{g}}$ and hence $U({\mathfrak{g}})$ acts on
$W^{\omega}$. Let $u\mapsto u^{\prime}$ be the principal anti-automorphism of
$U({\mathfrak{g}})$ extending $X\mapsto-X$ on ${\mathfrak{g}}$, and define a
pairing between $\pi^{\omega}$ and $U({\mathfrak{g}})\otimes\tau$ by
$\left\langle f,u\otimes v\right\rangle :=\left\langle (u^{\prime
}f)(e),v\right\rangle $. The pairing descends to $U({\mathfrak{g}}%
)\otimes_{U({\mathfrak{h}})}\tau$ in the second variable since if $X$ is in
${\mathfrak{h}}$ then%
\[
\left\langle f,uX\otimes v\right\rangle =-\left\langle \left[  \pi^{\omega
}\left(  X\right)  u^{\prime}f\right]  \left(  e\right)  ,v\right\rangle
=-\left\langle \widetilde{\sigma}^{\omega}\left(  X\right)  \left[  u^{\prime
}f\left(  e\right)  \right]  ,v\right\rangle =\left\langle u^{\prime}f\left(
e\right)  ,\tau\left(  X\right)  v\right\rangle =\left\langle f,u\otimes
Xv\right\rangle
\]

Let us check that the pairing is non-degenerate in the first variable. If $f$
lies in the left kernel, then $\left\langle (u^{\prime}f)(e),v\right\rangle
=0$ for any $u$, $v$ and hence $f$ vanishes at $e$ together with all its
derivatives. By Theorem \ref{thm:Pou}, $f$ is an analytic function and
therefore $f$ vanishes in the connected component of $e$. Since the support of
$f$ is $H$-invariant and $G/H$ is connected this implies $f=0$.

Now if the pairing is degenerate in the second variable, we quotient
$U({\mathfrak{g}})\otimes_{U({\mathfrak{h}})}\tau$ by the right kernel.
\end{proof}



\subsection{Gelfand-Kirillov dimension}

\label{subsec:Filt}

Let $A$ be a finitely generated associative algebra over $\C$ with increasing filtration $\left\{
F^{i}A \colon i\geq0\right\}  $, and let $M$ be a finitely generated $A$-module.
Choose a finite dimensional generating subspace $S$, define a filtration of
$M$ by $F^{i}M:=\left(  F^{i}A\right)  S$. It is known that there exists a polynomial $p$ such that $p(i):=\dim F^{i}M$ for large enough $i$,
and that the degree of $p$ does not depend on the choice of the
finite dimensional generating subspace $S$. This degree is called the Gelfand-Kirillov
dimension of $M$ and denoted by $\operatorname{GKdim}(M)$.

We will apply this in particular to the case when $A$ is the universal
enveloping algebra $U({\mathfrak{g}})$ for some complex Lie
algebra\ ${\mathfrak{g}}$, equipped with the usual filtration inherited from
the tensor algebra $T\left(  {\mathfrak{g}}\right)  $.

\begin{lem}
[\cite{Vog}, Lemma 2.3]\label{lem:VogTensor} Let ${\mathfrak{h}}%
\subset{\mathfrak{g}}$ be Lie algebras, and let $M$ be a finitely generated
$U({\mathfrak{h}})$ module. Then $N=U({\mathfrak{g}})\otimes_{U({\mathfrak{h}%
})}M$ is a finitely generated $U({\mathfrak{g}})$-module, and we have
\[
\operatorname{GKdim}(N)=\operatorname{GKdim}(M)+\dim({\mathfrak{g}%
}/{\mathfrak{h}})
\]

\end{lem}

If $\left(  \sigma,M\right)  $ is a finitely generated ${\mathfrak{g}}%
$-module, we will sometimes write $\operatorname{GKdim}(\sigma)$ for
$\operatorname{GKdim}(M)$ etc. Recall the annihilator variety $\mathcal{V}%
(\sigma)\subset{\mathfrak{g}}^{\ast}$ as defined in \S\S\S \ref{subsubsec:AnnVar}. It is easy to
show that
\[
\operatorname{GKdim}(\sigma)\leq\dim\mathcal{V}(\sigma).
\]
There is also a bound in the other direction for the Lie algebras we consider
in this paper.

\begin{thm}
[Gabber-Joseph]\label{thm:GabJo} Let ${\mathfrak{g}}$ be the Lie algebra of an
algebraic group over an algebraically closed field. Let $M$ be a finitely
generated $U({\mathfrak{g}})$ module. Then
\[
\operatorname{GKdim}(M)\geq\frac{1}{2}\dim\mathcal{V}(M).
\]

\end{thm}

For proof see \cite[Proposition 6.1.4]{Jos81} or \cite[Theorem
9.11]{KrLe}. This theorem does not hold for general Lie algebras.
On the other hand a stronger result is true for reductive Lie
algebras.

\begin{thm}
[\cite{Vog}, Theorem 1.1]\label{thm:Vog} Let $G$ be a real reductive group,
${\mathfrak{g}}$ be its complexified Lie algebra and $K$ be a maximal compact subgroup. Let $\left(  \sigma
,M\right)  $ be an irreducible Harish-Chandra module over $(\g,K)$. Then
\[
\operatorname{GKdim}(\sigma)=\frac{1}{2}\dim\mathcal{V}(\sigma).
\]

\end{thm}

%
%
%

%
\section{Proof of Theorem \ref{thm:Main}}
\label{sec:Der}
\setcounter{subsubsection}{0}

We will prove the theorem by induction, using the following three lemmas that will be proven in the next three subsections.
\begin{lemma}
\label{lem:Func2FuncFin} Let $\pi\in{\widehat{G}}_{n}$ and let $d:=depth(\pi)$.
Let $\alpha=(n_{1},...,n_{k})$ be a composition of $n-d$ and $\beta
=(d,\alp)=(d,n_{1},n_{2},...,n_{k})$. Then we have a natural embedding
$Wh^*_{\alpha}(A\pi) \hookrightarrow
Wh^*_{\beta}(\pi)$.
\end{lemma}

\begin{lemma} \label{lem:DimOrbFin}
Let $\lambda,\mu$ be partitions of $n$ and $n-d$ respectively, and suppose that

\begin{enumerate}
\item $\mathcal{O}_{\left(  d,\mu\right)  }\subseteq\overline{\mathcal{O}%
_{\lambda}}.$

\item $\dim\mathcal{O}_{\lambda}\leq\dim\mathcal{O}_{\mu}+\left(  2n-d\right)
\left(  d-1\right).$
\end{enumerate}

Then $\lambda=\left(  d,\mu\right).$

\end{lemma}

\begin{lemma}
\label{lem:AVDerFin} Let $\pi\in{\widehat{G}}_{n}$ and let $d:=depth(\pi)$.
Then $$\dim \cV(\pi) \leq \dim\cV(A\pi)+(2n-d)(d-1).$$
\end{lemma}


%

\begin{proof}
[Proof of Theorem \ref{thm:Main}]
We prove the statement by induction on $n$. For $n=0,1$ there is
nothing to prove. Now take $n>1$ and suppose that the theorem holds true for
all $r<n$. Let $\pi\in{\widehat{G}}_{n}, \, d:=depth(\pi) \text{ and } \lambda:=AP(\pi).$

Let $\mu$ be the depth composition  of $A\pi$. By the induction hypothesis we know that $\mu$ is a partition of $n-d$, $\cV(A\pi)= \overline{\cO_{\mu}}$ and
$Wh^*_{\mu}(A\pi^{\infty})\neq0$.

Let $\beta=(d,\mu)$. Then $\beta$ is the depth composition of $\pi$.
From the induction hypothesis and Lemma \ref{lem:Func2FuncFin} we obtain  $Wh^*_{\beta}(\pi^{\infty})\neq0$. It suffices to show that $\beta=\lambda$.

By Corollary \ref{cor:Mat} we have $ \cO_{\beta} \subset \cV(\pi)=\overline{\cO_{\lambda}}$. By Lemma \ref{lem:AVDerFin} we have $\dim \cO_{\lambda} \leq \dim\mathcal{O}_{\mu}+\left(  2n-d\right)
\left(  d-1\right).$
Thus, by Lemma \ref{lem:DimOrbFin}, $\beta=\lambda$.
\end{proof}

\subsection{Proof of Lemma \ref{lem:Func2FuncFin}} \label{subsec:PfFunc2FuncFin}

Now we construct a functional on $\pi\in{\widehat{G}}_{n}$ from a functional
on its adduced representation.

Let $P_{(n-d,d)}=\left(  G_{n-d}\times G_{d}\right)  \ltimes N_{(n-d,d)}$ be the
maximal parabolic subgroup corresponding to the partition $\left(
n-d,d\right)  $ and define the subgroup $S_{n-d,d}:=\left(  G_{n-d}\times
N_{1^{d}}\right)  \ltimes N_{(n-d,d)}$ where $N_{1^{d}}$ is the nilradical of
the Borel subgroup of $G_{d}$.

\begin{lemma} \label{lem:A1Step}
Let $\pi\in\widehat{G_{n}}$, and let $d:=depth(\pi)$. Then $\pi|_{P_{n}%
}=Ind_{S_{n-d,d}}^{P_{n}}\left(  A\pi\otimes\psi\otimes1\right)  $ where
$\psi$ is a nondegenerate unitary character of $N_{1^{d}}$.
\end{lemma}

\begin{proof}
This follows from the definition of $A\pi$ by a straightforward argument
involving induction by stages.
\end{proof}

\begin{prop}
\label{prop:pi2Api}  Let $\pi\in\widehat{G_{n}}$, and let $d:=depth(\pi)$. Then
there exists an $S_{n-d,d}$-equivariant map from $\pi^{\infty}$ to
$(A\pi)^{\infty}\otimes|\det|^{(d-1)/2}$ with dense image.

\end{prop}

\begin{proof}
Note that for $g\in S_{n-d,d}, \,\Delta_{P_{n}}(g)=|\det(g)| \text{ and
}\Delta_{S_{n-d,d}}(g)=|\det(g)|^d$. By Corollary \ref{cor:IndFunc}, the
evaluation $f\mapsto f\left(  e\right)  $ is an $S_{n-d,d}$-equivariant map from
$(\pi|_{P_{n}})^{\infty}$ to $(A\pi)^{\infty}\otimes|\det|^{(d-1)/2}$.
Note that $A\pi$ is an irreducible unitary representation of $G_{d}$
and thus, by
\cite[Theorem 3.4.12]{Wal1}, $(A\pi)^{\infty}$ is an irreducible \Fre representation of $G_{d}$. Thus it is enough to
show that this map does not vanish on $\pi^{\infty}$. For that let $v\in
\pi^{\infty}$ be a non-zero vector. Then $v \in\pi|_{P_{n}}^{\infty
}=Ind_{S_{n-d,d}}^{P_{n}}\left(  A\pi\otimes\psi\otimes1\right) ^{\infty}$
defines a smooth function on $P_{n}$ that does not vanish at some point $p$.
Then $\pi(p)v$ does not vanish at $e$.
\end{proof}

This proposition immediately implies Lemma \ref{lem:Func2FuncFin}.

\subsection{Proof of Lemma \ref{lem:DimOrbFin} } \label{subsec:PfDimOrbFin}
It will be useful to have a second formula for $\dim\mathcal{O}_{\lambda}$
directly in terms of $\lambda$. We will also consider reorderings of $\lambda$.

%
\begin{lem}
\label{lem:DimOrb} If $\alpha$ is a composition of $n$ then we have
\begin{equation} \label{eq:DimOrb}
\dim\mathcal{O}_{\alpha}\geq n^{2}+n-2\sum\nolimits_{i}i\alpha_{i}.
\end{equation}
 Moreover,
equality holds if and only if $\alpha$ is a partition.
\end{lem}

\begin{proof}
Let $\lambda = \alp^{\geq}$ and let $\nu=\lambda^{t}$ be the transposed partition. \\
We first show that equality holds in \eqref{eq:DimOrb} if $\alp$
is a partition, i.e. $\alp=\lambda$. By Proposition
\ref{prop:DimOrbT} it suffices to prove
\[
\sum\nolimits_{j}\nu_{j}^{2}=2\sum\nolimits_{i}i\lambda_{i}-n.
\]
Consider the Young diagram of $\lambda$ and write the number $2i$ in every box
in the $i$-th row. Compute the sum of these numbers in two ways: a) adding
rows first, b) adding columns first. This gives
\[
2\sum\nolimits_{i}i\lambda_{i}=2\sum\nolimits_{j}\left(  1+\cdots+\nu
_{j}\right)  =\sum\nolimits_{j}\left(  \nu_{j}^{2}+\nu_{j}\right)
=n+\sum\nolimits_{j}\nu_{j}^{2},
\]
as needed.

Now, for general $\alp$ we have $\mathcal{O}_{\alpha}=\mathcal{O}_{\lambda}$. If $\alpha$
is not a partition then we have strict inequality in (\ref{eq:DimOrb}) since $\sum\nolimits_{i}%
i\alpha_{i}>\sum\nolimits_{i}i\lambda_{i}$.
\end{proof}

\begin{cor}
\label{dimmu} Let $\mu$ be a partition of $n-d$ and write $\alpha=\left(
d,\mu\right)  $ then we have
\[
\dim\mathcal{O}_{\alpha}\geq\dim\mathcal{O}_{\mu}+\left(  2n-d\right)  \left(
d-1\right)  \text{,}%
\]
and equality holds if and only if $\alpha=\left(  d,\mu\right)  $ is a partition (i.e. $d \geq \mu_1$).
\end{cor}

\begin{proof}
We observe that%
\[
\sum\nolimits_{j}j\alpha_{j}=d+\sum\nolimits_{i}\left(  i+1\right)  \mu
_{i}=n+\sum\nolimits_{i}i\mu_{i}%
\]
Let $r=n-d$ and apply Lemma \ref{lem:DimOrb} to get the following
inequality%
\begin{align*}
\dim\mathcal{O}_{\alpha}-\dim\mathcal{O}_{\mu}  &  \geq\left(  n^{2}%
+n-2\sum\nolimits_{j}j\alpha_{j}\right)  -\left(  r^{2}+r-2\sum\nolimits_{i}%
i\mu_{i}\right) \\
&  =\left(  n^{2}-n\right)  -\left(  r^{2}+r\right)  =\left(  n+r\right)
\left(  n-r-1\right)  =\left(  2n-d\right)  \left(  d-1\right)
\end{align*}

By Lemma \ref{lem:DimOrb} equality holds iff $\alpha$ is a partition.
\end{proof}

\begin{proof}[Proof of Lemma \ref{lem:DimOrbFin}]
Let us write $\alpha=\left(  d,\mu\right)  $. By assumptions (1) and (2) and
Corollary \ref{dimmu} we get
\[
\dim\mathcal{O}_{\alpha}\leq\dim\mathcal{O}_{\lambda}\leq\dim\mathcal{O}_{\mu
}+\left(  2n-d\right)  \left(  d-1\right)  \leq\dim\mathcal{O}_{\alpha
}\text{.}%
\]

Hence equality must hold throughout. This implies that $\dim\mathcal{O}%
_{\alpha}=\dim\mathcal{O}_{\lambda}$ and, by Corollary \ref{dimmu}, that $\alpha$
is a partition. Now assumption (1) implies that $\alpha=\lambda$.
\end{proof}

\subsection{Proof of Lemma \ref{lem:AVDerFin}} \label{subsec:PfAVDerFin}


First we want to prove that $\dim\mathcal{V}
(\pi|_{P_{n}}) \geq \dim\mathcal{V}(\pi)$. We will start with a geometric lemma.

\begin{lemma}
\label{lem:geo} Let $\mathcal{O}\subset\mathfrak{g}_{n}^{\ast}$ be a nilpotent
coadjoint orbit. Then there exists an open dense subset $U\subset\mathcal{O}$
such that the restriction to $U$ of the projection $pr:=pr_{\mathfrak{p}_{n}%
^{\ast}}^{\mathfrak{g}_{n}^{\ast}}$ is injective.
\end{lemma}

For the proof see \S\S\S  \ref{subsubsec:Pflem:geo}.

\begin{cor}
\label{cor:geo} Let ${\mathcal{V}}\subset\mathfrak{g}_{n}^{\ast}$ be the
closure of a nilpotent coadjoint orbit. Then $\dim pr_{\mathfrak{p}_{n}^{\ast
}}^{\mathfrak{g}_{n}^{\ast}}({\mathcal{V}})=\dim{\mathcal{V}}$.
\end{cor}

\begin{cor}
\label{cor:dimAVres} Let
$\pi \in \widehat{G_n}$. Then $\dim \cV(\pi|_{P_n}) \geq \dim \cV(\pi)$.
\end{cor}

\begin{proof}
Let
$I=Ann_{U(\mathfrak{g}_{n})}\pi^{\omega}$ and
$J=Ann_{U(\mathfrak{p}_{n})}(\pi|_{P_{n}})^{\omega}$.
Since $(\pi|_{P_{n}})^{\omega}\supset\pi^{\omega}$ we have $J\subset
Ann_{U(\mathfrak{p}_{n})}\pi^{\omega}=I\cap U(\mathfrak{p}_{n})$ and
hence we get
\[
{gr(J)}\subset gr(I\cap U(\mathfrak{p}_{n}))\subset{gr(I)
}\cap S(\mathfrak{p}_{n}).
\]
Since $\mathcal{V}(\pi)$ is the closure of a nilpotent coadjoint
orbit, by the previous Corollary we conclude
\[
\dim\mathcal{V}(\pi|_{P_n})\geq\dim \text{Zeroes} \left(  {gr(I)}\cap
S(\mathfrak{p}_{n})\right)  =\dim pr_{\mathfrak{p}_{n}^{\ast}}^{\mathfrak{g}%
_{n}^{\ast}}(\mathcal{V}(\pi))=\dim{\mathcal{V}(\pi).}%
\]

\end{proof}

Now we want to use Theorem \ref{thm:GabJo} to bound $\dim\mathcal{V}
(\pi|_{P_{n}})$. In order to do that we find a finitely generated
$U(\mathfrak{p}_{n})$-module which is non-degenerately paired with
$(\pi|_{P_{n}})^{\omega}$ and therefore has the same annihilator.

\begin{proof}[Proof of Lemma \ref{lem:AVDerFin}]

We will use Lemma \ref{lem:A1Step}.
Let $\sigma:=  A\pi\otimes\psi\otimes1$ and $\tau:=  |\det|^{(1-d)/2}(\overline{(A\pi)^{HC}}\otimes\psi\otimes 1)$ be representations of $S_{n-k,k}$ and $\fs_{n-k,k}$ respectively,
where $\overline{(A\pi)^{HC}}$ denotes the complex conjugate representation to $(A\pi)^{HC}$. Then $\tau$ is equivariantly non-degenerately
paired with $\widetilde{\sigma} = |\det|^{(d-1)/2} \sigma$.

By Lemma \ref{lem:A1Step}, $\pi|_{P_{n}%
}=Ind_{S_{n-k,k}}^{P_{n}}(\sigma)$ and thus, by Lemma
\ref{lem:IndPairing}, $(\pi|_{P_{n}})^{\omega}$ is equivariantly non-degenerately
paired with  a quotient of $U(\fp_n) \otimes _{U(\fs_{n-k,k})} \tau$. Denote this quotient by $L$. Then  $\cV(\pi|_{P_{n}}) = \cV(L)$.
Note that twist by a character does not effect annihilator variety and Gelfand-Kirillov dimension. Neither does exterior tensor product with a character.

From Lemma \ref{lem:VogTensor}
we obtain :
\[
\operatorname{GKdim}(L)\leq \operatorname{GKdim}
(\tau)+\dim \fp_n - \dim \fs_{n-k,k} =\operatorname{GKdim}(A\pi^{HC})+(2n-d)(d-1)/2.
\]
By Theorem \ref{thm:GabJo} we have $\dim\mathcal{V}(L) \leq
2\operatorname{GKdim}(L)$. By Theorem \ref{thm:Vog} we have $2 \operatorname{GKdim}(A\pi^{HC}) = \dim\mathcal{V}(A\pi)$.
By Corollary \ref{cor:dimAVres} we have
$\dim\mathcal{V}(\pi)\leq\dim\mathcal{V} (\pi|_{P_{n}})$. Altogether we
have
\begin{multline*}
\dim\mathcal{V}(\pi)\leq\dim\mathcal{V} (\pi|_{P_{n}})=\dim\mathcal{V}
(L)\leq2\operatorname{GKdim}(L)\leq\\
\leq2\operatorname{GKdim}(A\pi^{HC})+(2n-d)(d-1)=\dim\mathcal{V} (A\pi
)+(2n-d)(d-1).
\end{multline*}

\end{proof}

\subsubsection{Proof of Lemma \ref{lem:geo} } \label{subsubsec:Pflem:geo}

Identify $\mathfrak{g}_{n}^{\ast}$ with $\mathfrak{g}_{n}$ using the trace
form; then $\mathfrak{p}_{n}^{\ast}$ consists of matrices whose last
\emph{column} is zero, and $pr\left(  A\right)  $ replaces the last column of
$A$ by zero. Now $\mathcal{O}$ corresponds to a nilpotent orbit in
$\mathfrak{g}_{n}$; let $k$ denote the size of the biggest Jordan block in
$\mathcal{O}$ and define
\[
U:=\{A\in\mathcal{O}\mid e_{n}^{\prime}A^{k-1}\neq0\}\text{ where }%
e_{n}^{\prime}=\left(  0,\cdots,0,1\right)
\]
Then $U$ is an open dense subset of $\mathcal{O}$, and we will show $pr|_{U}$
is injective. Suppose $A,B\in U$ with $pr\left(  A\right)  =pr\left(
B\right)  $, then $A$ and $B$ differ only in the last column and so
\[
B=A+v e_{n}^{\prime}
\]
where $v$ is some column vector; it suffices to prove that $v=0$.

We first prove by induction that $e_{n}^{\prime}A^{i}v=0$ for any $i\geq0$.
Since $A$ and $B$ are nilpotent we have%
\[
0=\operatorname{Tr}B=\operatorname{Tr}A+\operatorname{Tr}v e_{n}^{\prime}
=e_{n}^{\prime}v
\]
which proves the claim for $i=0$. Now suppose the claim holds for $i<l$, then%
\begin{equation}
B^{l+1}=(A+v e_{n}^{\prime})^{l+1}=A^{l+1}+
\sum\nolimits_{j=0}^{l} A^{j}\left(  v e_{n}^{\prime}\right)  A^{l-j}+\cdots\label{=Bl}
\end{equation}
Each omitted term in (\ref{=Bl}) has at least two factors of the form $v
e_{n}^{\prime}$, hence at least one factor of the form $e_{n}^{\prime}%
A^{i}v$ for some $0 \leq i<l$ , which vanishes by the induction hypothesis.
Now taking trace in (\ref{=Bl}) we get%
\[
0=
\sum\nolimits_{j=0}^{l}
\operatorname{Tr}\left(  A^{j}v e_{n}^{\prime}A^{l-j}\right)  =\left(
l+1\right)  e_{n}^{\prime}A^{l}v
\]
which implies the claim for $i=l$, and by induction for all $i$.

Suppose now by way of contradiction that $v\neq0$ and let $m\geq0$ be the
largest integer such that $A^{m}v\neq0$. Substitute $l=k-1$ in (\ref{=Bl});
since $A^{k}=B^{k}=0$ we get
\begin{equation}
0=\sum\nolimits_{j=0}^{k-1}
A^{j}v e_{n}^{\prime}A^{k-1-j}=%
\sum\nolimits_{j=0}^{m}
A^{j}v e_{n}^{\prime}A^{k-1-j} \label{=Bk}%
\end{equation}
Note that $v,Av,...,A^{m}v$ are linearly independent; indeed suppose
$\sum\nolimits_{j=i}^{m}
c_{j}A^{j}v=0$ with $c_{i}\neq0$, then multiplying by $A^{m-i}$ we deduce
$c_{i}A^{m}v=0$, which is a contradiction. Therefore we can choose a row
vector $\phi$ such that $\phi A^{m}v=1$ but $\phi A^{j}%
v=0$ for any $j<m$. Multiplying (\ref{=Bk}) by $\phi$ on the left we
get $0=e_{n}^{\prime}A^{k-1-m}$ which contradicts the assumption that
$e_{n}^{\prime}A^{k-1}\neq0$.
\proofend

\begin{remark}
Using 
\cite[3.1;4.1-4.2]{Ber} one can show that $U$ is a single $P_n$-orbit. 
\end{remark}



\section{Computation of adduced representations for (almost) all unitary
representations}

\label{sec:CompDer}

It is an interesting and important problem to explicitly compute the adduced representations of all irreducible unitary representations of $G_n$. The answer has been conjectured in \cite{SaSt} and in
the present paper we make substantial progress towards the proof of this conjecture.

\subsection{Vogan Classification} \label{subsec:VogClas}

By the Vogan classification \cite{Vog-class}, irreducible unitary
representations of $G_n$ are Bernstein-Zelevinsky
(BZ) products of the form
\[
\pi=\pi_{1}\times\cdots\times\pi_{k}%
\]
where each $\pi_{i}$ is one of the following basic unitary representations:

\begin{enumerate}[(a)]
\item \label{it:Char}\emph{A one-dimensional unitary character of some }$G_m$. Such a character is of the form
\[
x\mapsto\left(  \mathrm{sgn}\det x\right)  ^{\varepsilon}\left\vert \det
x\right\vert ^{z},\varepsilon\in\left\{  0,1\right\}  ,z\in\mathbb{C}%
\]
and we shall denote it by $\chi\left(  m,\varepsilon,z\right)  $. This
character is unitary if $z$ is imaginary, i.e.
\[
z=it,t\in\mathbb{R}.
\]

\item \label{it:StCS}\emph{A Stein complementary series representation of some }$G_{2m}$, twisted by a unitary character. The Stein
representations are complementary series of the form%
\[
\sigma\left(  2m,s\right)  =\chi\left(  m,0,s\right)  \times\chi\left(
m,0,-s\right), s\in\left(  0,1/2\right)
\]
and we write $\sigma\left(  2m,s;\varepsilon,it\right)  $ to denote its twist
by$\chi\left(  2m,\varepsilon,it\right)  $.

\item \label{it:Speh} \emph{A Speh representation of some }$G_{2m}$,
twisted by a unitary character. As shown in \cite{BSS} and \cite{SaSt} the
Speh representation $\delta\left(  2m,k\right)  $ is the unique irreducible
submodule of%
\begin{equation}
\chi\left(  m,0,k/2\right)  \times\chi\left(  m,\varepsilon_{k+1},-k/2\right)
,k\in\mathbb{N},\varepsilon_{k+1}\equiv k+1\left(  \operatorname{mod}2\right)  \label{=BSS}%
\end{equation}

and we write $\delta\left(  2m,k;it\right)  $ to denote its twist by
$\chi\left(  2m,0,it\right)  $.

\item  \label{it:SpehCS} \emph{A Speh complementary series representation of some }$G_{4m}$, twisted by a unitary character. The Speh
complementary series representation is%
\[
\psi\left(  4m,k,s\right)  =\delta\left(  2m,k;0,s\right)  \times\delta\left(
2m,k;0,-s\right), s\in\left(  0,1/2\right)
\]
and we write $\psi\left(  4m,k,s;it\right)  $ to denote its twist by
$\chi\left(  4m,0,it\right)  $.
\end{enumerate}

The reader might well ask why we do not consider twists in (\ref{it:Speh}) and (\ref{it:SpehCS}) for
$\varepsilon=1$. The reason is the following:

\begin{lemma}
Speh representation and their complementary series are \emph{unchanged} if we
twist them by the sign character.
\end{lemma}

\begin{proof}
If $\pi$ is a representation of $G_n$ we denote
its sign twist by $\widetilde{\pi}=\pi\otimes\chi\left(  n,1,0\right)
$. This operation is compatible with parabolic induction in the sense that we
have%
\begin{equation}
\widetilde{\pi_{1}\times\pi_{2}}=\widetilde{\pi_{1}}\times\widetilde{\pi_{2}%
}. \label{ds}%
\end{equation}

We leave the easy verification of (\ref{ds}) to the reader.

For the group $GL\left(  2,\mathbb{R}\right)  $, the Speh representations
$\delta\left(  2,k\right)  $ are precisely the discrete series. In this case
the result $\delta\left(  2,k\right)  =\widetilde{\delta\left(  2,k\right)  }$
is well known (see 1.4.7 in \cite{Vog-book}). The general Speh representation
$\delta\left(  2m,k\right)  $ is the unique irreducible quotient (Langlands
quotient) of%
\begin{equation}
\delta\left(  2,k;s_{1}\right)  \times\cdots\times\delta\left(  2,k;s_{m}%
\right)  \text{ where }s_{i}=\frac{m+1}{2}-i \label{=Speh}%
\end{equation}
see, e.g. \cite{Vog-class}. By (\ref{ds}) the induced representation
(\ref{=Speh}) is unchanged under sign twist. Therefore so is its unique
irreducible quotient.

The result for Speh complementary series now follows from (\ref{ds}).
\end{proof}

\subsection{Annihilator
variety and rank}

\label{subsec:CompAV}


In this subsection we compute the associated partition in terms of the Vogan classification. Let
\begin{equation}
\pi=\prod\nolimits_{i=1}^{k}\chi_{i}\times\prod\nolimits_{j=1}^{l}\delta
_{j}\text{ with }n=\sum\nolimits_{i=1}^{k}p_{i}+2\sum\nolimits_{j=1}^{l}q_{j}
\label{=pi}%
\end{equation}
where $\chi_{i}$ is a character of $G_{p_{i}}$ and $\delta_{j}$ is a Speh
representation of $G_{2q_{j}}$ (perhaps twisted by a nonunitary character in
order to include complementary series). By Corollary \ref{cor:BarBoz} , to
compute $AP(\pi)  $ it suffices to determine
$AP(  \chi_{i})  $ and $AP(\delta
_{j})$. Clearly $AP(  \chi_{i})=1^i$. By \cite[Theorem 3]{SaSt}, the adduced of the Speh representation $\delta(2n,k)\in{\widehat{G}}_{2n}$ is
the Speh representation $\delta(2(n-1),k)\in{\widehat{G}}_{2n-2}$. By induction,
we obtain from Theorem \ref{thm:Main} that $AP(\delta_{j})  = 2^j.$


\begin{thm}
\label{thm:AVTadic} For $\pi$ as in (\ref{=pi}), $AP(\pi)^{t}$ has one part
of size $p_{i}$ for each $i$ and two parts
of size $q_{j}$ for each $j$. Consequently the rank (see Definition \ref{def:rank}) and the Gelfand-Kirillov dimension  of $\pi$  are given by
\begin{equation}
{\operatorname{rk}}(\pi)=n-\max(p_{i},q_{j}), \quad \operatorname{GKdim}
(\pi)=\frac{1}{2}\left(  n^{2}-\sum p_{i}^{2}-2\sum q_{j}^{2}\right)
\end{equation}
\end{thm}

\begin{cor}
\label{cor:SmallRep} Let $\pi\in{\widehat{G}}_{n}$ and let $k<n/2$. Then
${\operatorname{rk}}(\pi)=k$ if and only if there
exist a representation $\sigma\in{\widehat{G}}_{k}$ and a character $\chi
\in{\widehat{G}}_{n-k}$ such that $\pi=\sigma\times\chi$.
\end{cor}

\begin{remark}\label{rem:SmallRep}
By \cite[Part II, Corollary 3.2]{Sca}, an analogous statement holds for Howe
rank. This gives an independent proof of the statement that for any $\pi
\in{\widehat{G}}_{n}$, $$\text{ Howe rank }(\pi)=\min({\operatorname{rk}}(\pi),\lfloor
n/2\rfloor).$$
\end{remark}

The following theorem follows from \cite{BSS, SaSt}.

\begin{thm}
For the Speh representation $\delta(2n,k)\in{\widehat{G}}_{2n}$ there exist
degenerate principal series representations $\pi_{k}$ and $\pi_{-k}$, induced
from certain characters of the standard parabolic subgroup given by the
partition $(n,n)$, an embedding $i \colon \delta(2n,k)\hookrightarrow\pi_{k}$ and an
epimorphism $p \colon \pi_{-k}\twoheadrightarrow\delta(2n,k)$.
\end{thm}

\begin{cor}
\label{cor:UnEmb} Any irreducible unitarizable representation $\pi$ of $G_{n}$
can be presented both as a subrepresentation and as a quotient of a degenerate
principal series representation with the same annihilator variety. Those
degenerate principal series representations will be induced from characters of
the standard parabolic described by the partition which is transposed to the
partition describing $\mathcal{V}(\pi)$.
\end{cor}

\begin{rem}
In \cite{AGS} we show, using \cite{BSS}, that all other Jordan-Holder
constituents of the degenerate principal series representations mentioned
above will have smaller annihilator varieties.
\end{rem}


\subsection{Infinitesimal characters - general considerations}

In this subsection only, we let $\mathfrak{g}$ denote an arbitrary complex
reductive Lie algebra. We fix a Cartan subalgebra $\mathfrak{h}$ and a Borel
subalgebra $\mathfrak{b}=\mathfrak{h}+\mathfrak{n}$, and let $W$ denote the
Weyl group of $\mathfrak{g}$. Let $Z\left(  \mathfrak{g}\right)  $ denote the
center of the universal enveloping algebra $U\left(  \mathfrak{g}\right)  $.
Then by the Harish-Chandra homomorphism we have $Z\left(  \mathfrak{g}\right)
\approx S\left(  \mathfrak{h}\right)  ^{W}$. Thus each $\lambda\in
\mathfrak{h}^{\ast}$ determines a character $\chi_{\lambda}$ of $Z\left(
\mathfrak{g}\right)  $ with $\chi_{w\lambda}=\chi_{\lambda}$ for all $w\in W$.

The Harish-Chandra homomorphism is a special case of the following more
general construction. Let $\mathfrak{q}=\mathfrak{l}+\mathfrak{u}$ be a
standard parabolic subalgebra containing $\mathfrak{b}$ so that
$\mathfrak{h\subset l}$ and $\mathfrak{u\subset n}$. Then we have a triangular
decomposition $\mathfrak{g}=\mathfrak{u}+\mathfrak{l}+\overline{\mathfrak{u}}%
$, and by the PBW theorem the universal enveloping algebra of $\mathfrak{g}$
can be decomposed as follows:%
\begin{equation}
U\left(  \mathfrak{g}\right)  =U\left(  \mathfrak{u}\right)  \otimes U\left(
\mathfrak{l}\right)  \otimes U\left(  \overline{\mathfrak{u}}\right)
=U\left(  \mathfrak{l}\right)  \oplus\left[  \mathfrak{u}U\left(
\mathfrak{g}\right)  \mathfrak{+}U\left(  \mathfrak{g}\right)  \overline
{\mathfrak{u}}\right]  \label{=triang}%
\end{equation}
Let $\mathcal{P=P}_{\mathfrak{l}}^{\mathfrak{g}}$ denote the corresponding
projection from $U\left(  \mathfrak{g}\right)  $ to $U\left(  \mathfrak{l}%
\right)  $, and let $Z\left(  \mathfrak{g}\right)  $ and $Z\left(
\mathfrak{l}\right)  $ denote the centers of $U\left(  \mathfrak{g}\right)  $
and $U\left(  \mathfrak{l}\right)  $.

\begin{lemma}
\begin{enumerate}
\item $\mathcal{P}$ \ is $ad\left(  \mathfrak{l}\right)  $-equivariant.

\item $\mathcal{P}$ maps $Z\left(  \mathfrak{g}\right)  $ to $Z\left(
\mathfrak{l}\right)  $.

\item For $z\in Z\left(  \mathfrak{g}\right)  $ \ we have $z-\mathcal{P}%
\left(  z\right)  \in\mathfrak{u}U\left(  \mathfrak{g}\right)  $.
\end{enumerate}
\end{lemma}

\begin{proof}
See e.g. \cite{Vog-book} P. 118.
\end{proof}

\begin{lemma}
Let $V$ be a $\mathfrak{g}$-module, then

\begin{enumerate}
\item The subspace $\mathfrak{u}V$ is $Z\left(  \mathfrak{g}\right)
$-invariant and $U\left(  \mathfrak{l}\right)  $-invariant.

\item The quotient space $V/\mathfrak{u}V$ is a $Z\left(  \mathfrak{g}\right)
$-module and a $U\left(  \mathfrak{l}\right)  $-module.

\item For $z\in Z\left(  \mathfrak{g}\right)  $ the actions of $z$ and
$\mathcal{P}\left(  z\right)  $ agree on $V/\mathfrak{u}V$.
\end{enumerate}
\end{lemma}

\begin{proof}
Parts 1 and 2 are straightforward. Part 3 follows from the previous lemma.
\end{proof}

We say that a $\mathfrak{g}$-module $V$ has infinitesimal character $\lambda$
if each $z\in Z\left(  \mathfrak{g}\right)  $ acts by the scalar
$\chi_{\lambda}\left(  z\right)  $. We say that $V$ has generalized
infinitesimal character $\chi_{\lambda}$ if there is an integer $n$ such that
$\left(  z-\chi_{\lambda}\left(  z\right)  \right)  ^{n}$ acts by $0$ for all
$z\in Z\left(  \mathfrak{g}\right)  $. We say that $V$ is $Z\left(
\mathfrak{g}\right)  $-finite if $V$ is annihilated by an ideal of finite
codimension in $Z\left(  \mathfrak{g}\right)  $. If $V$ is $Z\left(
\mathfrak{g}\right)  $-finite then $V$ decomposes as a finite direct sum
\[
V=V_{1}\oplus\cdots\oplus V_{k}%
\]
where each $V_{i}$ has generalized infinitesimal character.

\begin{cor}
[Casselman - Osborne]\label{CO} Let $V$ be a $Z\left(  \mathfrak{g}\right)
$-finite $\mathfrak{g}$-module, then $V/\mathfrak{u}V$ is a $Z\left(
\mathfrak{l}\right)  $-finite $\mathfrak{l}$-module. Moreover if the
generalized infinitesimal character $\chi_{\mu}$ occurs in $V/\mathfrak{u}V$,
then there exists $\lambda\in\mathfrak{h}^{\ast}$ such that

\begin{enumerate}
\item the generalized infinitesimal character $\chi_{\lambda}$ occurs in $V$;

\item $\mu=\lambda+\rho\left(  \mathfrak{u}\right)  $ where $\rho\left(
\mathfrak{u}\right)  $ is half the sum of the roots of $\mathfrak{h}$ in
$\mathfrak{u.}$
\end{enumerate}
\end{cor}

\begin{proof}
The proof is similar to that of Corollary 3.1.6 in \cite{Vog-book}.
\end{proof}

\subsection{Infinitesimal characters for $G_n$}

If $\pi$ is an irreducible admissible representation of $G_n$, the infinitesimal character $\xi_{\pi}$ of $\pi$ can
be regarded as a multiset (set with multiplicity) of $n$ complex numbers
$\left\{  z_{1},\ldots,z_{n}\right\}  $. If all the $z_{i}$ are real we say
that $\pi$ has real infinitesimal character. For convenience we write $\sqcup$
for the (disjoint) union of multisets; e.g. $\left\{  1,2\right\}
\sqcup\left\{  2,3\right\}  =\left\{  1,2,2,3\right\}  .$

For $m\in\mathbb{N}$ and $z\in\mathbb{C}$ we define the corresponding
``segment" to be the set
\[
S\left(  m,z\right)  =\left\{  z_{1},\cdots,z_{m}\right\}  \text{ where }%
z_{i}=z+ \left(  m+1\right)  /2-k.%
\]
Thus $z_{1},\cdots,z_{m}$ is an arithmetic progression of length $m$, mean
$z$, and common difference $1$. 

The following lemma summarizes the main facts about infinitesimal characters
for unitary representations of $G_n $.

\begin{lemma}
\label{lem:infchar}

\begin{enumerate}
\item $\xi_{\pi_{1}\times\cdots\times\pi_{k}}=\xi_{\pi_{1}}\sqcup\cdots
\sqcup\xi_{\pi_{k}}$

\item For $\pi=\chi\left(  m,\varepsilon,z\right)  $ we have $\xi_{\pi
}=S\left(  m,z\right)  $

\item For $\pi=\sigma\left(  2m,s;\varepsilon,it\right)  $ we have $\xi_{\pi
}=S\left(  m,s+it\right)  \sqcup S\left(  m,-s+it\right)  $

\item For $\pi=\delta\left(  2m,k;it\right)  $ we have $\xi_{\pi}=S\left(
m,\left(  k/2\right)  +it\right)  \sqcup S\left(  m,\left(  -k/2\right)
+it\right)  $

\item For $\pi=\psi\left(  4m,k,s;it\right)  $ we have $\xi_{\pi}=\bigsqcup
S\left(  m,\pm\left(  k/2\right)  \pm s+it\right)  $
\end{enumerate}
\end{lemma}

\begin{proof}
1) and 2) are standard \cite{Vog-book} and together they imply 3). Similarly
1) and 4) imply 5). Part 4) follows from 1) and formula (\ref{=BSS});
alternatively one may deduce it from 1), formula (\ref{=Speh}), and the fact
that the discrete series $\delta\left(  2,k\right)  $ of $GL\left(
2,\mathbb{R}\right)  $ has infinitesimal character $\left\{  \frac{k}%
{2},-\frac{k}{2}\right\}  $.
\end{proof}

\begin{lemma}
If $\pi\in \widehat{G_n}$ then $\xi_{\pi}$ is
\emph{symmetric} in the sense $z$ and $-\bar{z}$ have the same multiplicity in
$\xi_{\pi}$ for all $z\in\mathbb{C}$.
\end{lemma}

This lemma is in fact an easy elementary fact about Hermitian
representations. For $\pi\in \widehat{G_n}$ it also follows easily from the previous lemma by checking it for basic representations.




\subsection{Adduced representations}

\begin{conjecture}
[\cite{SaSt}]\label{conj:SpehCS} Let $\pi\in\widehat{G_{n}}$ and write
$\pi=\pi_{1}\times\cdots\times\pi_{k}$ as in the Vogan classification with
each $\pi_{i}$ a basic unitary representation of type \ref{it:Char}%
-\ref{it:SpehCS} listed above. Then we have
\begin{equation}
A\pi=A\pi_{1}\times\cdots\times A\pi_{k} \label{=dV}%
\end{equation}
where
\begin{align}
A\left(  \chi\left(  m,\varepsilon,it\right)  \right)   &  =\chi\left(
m-1,\varepsilon,it\right) \label{=dB}\\
A\left(  \sigma\left(  2m,s;\varepsilon,it\right)  \right)   &  =\sigma\left(
2\left(  m-1\right)  ,s;\varepsilon,it\right) \nonumber\\
A\left(  \delta\left(  2m,k;it\right)  \right)   &  =\delta\left(  2\left(
m-1\right)  ,k;it\right) \nonumber\\
A\left(  \psi\left(  4m,k,s;it\right)  \right)   &  =\psi\left(  4\left(
m-1\right)  ,k,s;it\right) \nonumber
\end{align}

\end{conjecture}

Most of this result is already known. The identity (\ref{=dV}) is proved in
\cite{Sahi-Kirillov}. As for (\ref{=dB}), part 1 is obvious, part 2 is proved
in \cite{Sahi-PAMS},
part 3 is proved in \cite{SaSt}, where part 4 was conjectured. We show that
the techniques of the present paper suffice to prove (\ref{=dB}) part 4 for
$k\neq m$.

\begin{lemma}
Let $\pi\in\widehat{G_{n}}$ and for $t\in\mathbb{R}$ let $\pi[
it]  $ denote the unitary twist of $\pi$ by the character $\left\vert
\det\right\vert ^{it}$. Then we have $A\left(  \pi[it]  \right)
=\left(  A\pi\right)[it]  $.
\end{lemma}

\begin{proof}
This is straightforward.
\end{proof}

\begin{lemma}
\label{lem:CharOrbUn} The Speh complementary series representation
$\psi\left(  4m,k,s;it\right)  $ is uniquely determined by its infinitesimal
character and associated partition.
\end{lemma}

\begin{proof}
By Lemma \ref{lem:infchar} $\pi=\psi\left(  4m,k,s;it\right)  $ has
infinitesimal character $\xi_{\pi}=\bigsqcup S\left(  m,\pm\left(  k/2\right)
\pm s+it\right)  $, and by Theorem
\ref{thm:AVTadic} its associated partition is $4^{m}$. Since $0<s<1/2$ it follows that

\begin{enumerate}
\item $2\operatorname{Re}\left(  z\right)  $ is not an integer for any
$z\in\xi_{\pi}$,

\item $\max\left\{  2\operatorname{Re}\left(  z\right)  :z\in\xi_{\pi
}\right\}  =m-1+k+2s$
\end{enumerate}

Let $\pi^{\prime}$ be a unitary representation with the same infinitesimal
character and associated partition as $\pi$. Write $\pi^{\prime}=\pi_{1}\times\cdots\times
\pi_{l}$ as in the Vogan classification. Then condition 1 above implies that
none of the $\pi_{i}$ can be unitary characters or Speh representations, while
condition 2 implies that $\pi^{\prime}$ cannot be the product of two Stein
representations of $G_{2m}$, for then we would have $\max\left(
2\operatorname{Re}\left(  z\right)  \right)  <m$. Therefore we conclude that
$\pi^{\prime}$ is a Speh complementary series representation of $G_{4m}$. Thus
$\pi^{\prime}=\psi\left(  4m,k^{\prime},s^{\prime};it^{\prime}\right)  $. By
looking at the integral and fractional parts of $\max\left(
2\operatorname{Re}\left(  z\right)  \right)  $ we deduce $k=k^{\prime}$ and
$s=s^{\prime}$. By looking at the imaginary part of $z\in\xi_{\pi}=\xi
_{\pi^{\prime}}$ we conclude that $t=t^{\prime}$.
\end{proof}

\begin{remark}
Due to the previous lemma, Conjecture \ref{conj:SpehCS} becomes
now equivalent to the statement that for any
$\pi\in\widehat{G_{n}}$, the infinitesimal character parameter of
$A\pi$ is obtained from the infinitesimal character parameter of
$\pi$ by the following procedure: consider the Young diagram $X$
with the sizes of rows described by the partition corresponding to
$\pi$, write the infinitesimal character parameter of $\pi$ in the
columns such that in each column we will have a segment, and make
each of those segments shorter by one without changing its center.
This is similar to the effect of highest derivatives on the
Zelevinsky classification in the $p$-adic case (see \S
\ref{sec:padic_short}). We cannot prove this statement here, but
we deduce its weaker version from Corollary \ref{CO} and
Proposition \ref{prop:pi2Api}.
\end{remark}

\begin{proposition}
\label{prop:CassOS} Let $\pi\in\widehat{G_{n}}$ and let $d=depth(\pi)$. Let
$S=\left\{  z_{1},\ldots,z_{n}\right\}  $ and $S^{\prime}=\left\{
y_{1},\ldots,y_{n-d}\right\}  $ be the multisets corresponding to
infinitesimal characters of $\pi$ and $A\pi$ respectively. Then $S^{\prime}$is
obtained from $S$ by deleting $d$ of the\ $z_{i}$'s and adding $1/2$ to
each of the remaining $z_{i}$'s.
\end{proposition}

\begin{proof}
Let $\sigma:= A\pi$ and let $\pi^{\infty}$ and $\sigma^{\infty}$ be the spaces
of smooth vectors of $\pi$ and $\sigma$, respectively, and let $\lambda$ and
$\mu$ be their infinitesimal character parameters.

By Proposition \ref{prop:pi2Api} there is an $S_{n-d,d}$-equivariant morphism
$\varphi:\pi^{\infty}\rightarrow\sigma^{\infty}\otimes|\det|^{(d-1)/2}$ with
dense image. Since $N_{n-d,d}$ acts trivially on $\sigma^{\infty}\otimes
|\det|^{(d-1)/2}$, $\varphi$ factors through the quotient $\pi^{\infty
}/\mathfrak{u}\pi^{\infty},$ where $\mathfrak{u}=Lie\left(  N_{n-d,d}\right)
$. By Corollary \ref{CO}, the possible $G_{n-d}\times G_{d}$ infinitesimal
characters in $\pi^{\infty}/\mathfrak{u}\pi^{\infty}$ are of the form
$w\lambda+\rho\left(  \mathfrak{u}\right)  $. Further restricting to $G_{n-d}$
we conclude that these characters are of the form%
\[
\lambda^{-}+\frac{d}{2}\left(  1,1,\cdots,1\right)
\]
where $\lambda^{-}\subset\lambda$ is a subset of size $n-d$.

On the other hand, the $G_{n-d}$ module $\sigma^{\infty}\otimes|\det
|^{(d-1)/2}$ has infinitesimal character
\[
\mu+\frac{d-1}{2}\left(  1,1,\cdots,1\right)  .
\]

Comparing the two displayed expressions we get $\mu=\lambda^{-}+\frac{1}%
{2}\left(  1,1,\cdots,1\right)  $ and the result follows.
\end{proof}

We are now ready to prove the main theorem of this section.

\begin{theorem}
\label{thm:Apsi} Suppose $k\neq m$ then $A\left(  \psi\left(
4m,k,s;it\right)  \right)  =\psi\left(  4\left(  m-1\right)  ,k,s;it\right)  $.
\end{theorem}

\begin{proof}
Let $\pi=\psi\left(  4m,k,s;it\right)  ,\pi^{\prime}=\psi\left(  4\left(
m-1\right)  ,k,s;it\right)  $ and let $\xi,\xi^{\prime},\xi^{^{\prime\prime}}$
be the infinitesimal characters of $\pi,\pi^{\prime},A\pi$\ respectively. By
Theorem
\ref{thm:AVTadic}, $AP(\pi) = 4^m$ and $AP(\pi^{\prime})=4^{m-1}$. By Theorem \ref{thm:Main},
$AP(  A\pi) =4^{m-1}$ as well. Therefore
by Lemma \ref{lem:CharOrbUn} it suffices to show that $\xi^{\prime}
=\xi^{^{\prime\prime}}$.

For simplicity we assume $t=0$, the argument is the same if $t\neq0$. Now by
Lemma \ref{lem:infchar} we have%
\[
\xi=\bigsqcup S\left(  m,\pm\frac{k}{2}\pm s\right)  \text{ and }\xi^{\prime
}=\bigsqcup S\left(  m-1,\pm\frac{k}{2}\pm s\right)
\]

Now by \ref{prop:CassOS} $\xi^{^{\prime\prime}}$ is obtained from $B=\xi
+\frac{1}{2}$ by deleting $4$ elements. Indeed $\xi^{\prime}$ is obtained from
$B$ by deleting the following $4$ elements
\[
C=\{\frac{m}{2}\pm\frac{k}{2}\pm s\}.
\]
and we need to show that \emph{no }other infinitesimal character of a unitary
representation can be so obtained.

In fact we note that $\xi^{\prime}$ is the only symmetric submultiset of
$\xi+\frac{1}{2}$ with $\left\vert \xi^{\prime}\right\vert =4\left(
m-1\right)  $. Indeed any other symmetric subset of cardinality $|B|-4$ is
obtained from $B$ by replacing a symmetric subset of $\xi^{\prime}$ with a
symmetric subset of equal size contained in $C$, but if $k\neq m$ then $C$ has
no symmetric subsets.
\end{proof}

\begin{remark}
For $k=m$ the proof of Theorem \ref{thm:Apsi} fails, since in this case we
have%
\[
C=\{m\pm s,\pm s\}
\]
which admits the symmetric subset $\left\{  \pm s\right\}  .$ Indeed it is
easy to see that for $k=m$ the infinitesimal character of $\psi\left(
4m-4,m-1,\frac{1}{2}-s\right)  $ is also a subset of $\xi+\frac{1}{2}$. Therefore we may
conclude only that%
\[
A\psi\left(  4m,m,s\right)  =\psi\left(  4m-4,m,s\right)  \text{ or }%
\psi\left(  4m-4,m-1,\frac{1}{2}-s\right)
\]
An additional argument is required to rule out the latter possibility.
\end{remark}

\section{The complex case}

\label{sec:Complex}

Let us discuss the setting and the geometry of nilpotent orbits. We have
\[
G=G_{n}=\operatorname{GL}(n,{\mathbb{C}}),\quad{\mathfrak{g}}\approx
{\mathfrak{gl}}(n,{\mathbb{C}})\oplus{\mathfrak{gl}}(n,{\mathbb{C}}%
),\quad{\mathfrak{g}}_{\R}={\mathfrak{gl}}(n,{\mathbb{C}})\oplus0,\quad
{\mathfrak{k}}=\{(x,x)\},\quad{\mathfrak{k}^{\bot}}=\{(x,-x)\},
\]
where $\fk$ denotes the complexified Lie algebra of maximal
compact subgroup. The nilpotent orbits are parameterized by pairs
of partitions and have the form
${\mathcal{O}}_{{\lambda}}\times{\mathcal{O}}_{\mu}$. However,
associated orbits of Harish-Chandra modules intersect
$\mathfrak{k}^{\bot}$ and are therefore of the form
${\mathcal{O}}_{{\lambda}}\times{\mathcal{O}}_{{\lambda }}$ and so
are still parameterized by single partitions, rather than two
partitions. Standard parabolic subgroups are also parameterized by
single partitions. The degenerate Whittaker functionals are
defined in the same way as in \S\S\S \ref{subsec:genWhitFunc}:
using $J_{{\alpha}}\in{\mathfrak{g}}_{\R}$. Therefore they are
also parameterized by single partitions. Therefore the formulation
of Theorem \ref{thm:Main}
in the complex case stays the same. The proof of
Theorem \ref{thm:Main} is obtained from the proof in \S
\ref{sec:Der} by replacing the term ``coadjoint nilpotent orbit"
by ``coadjoint nilpotent orbit that intersects
$\mathfrak{k}^{\bot}$" and doubling all the expressions for the
dimensions of such orbits.

The Vogan classification of $\widehat{GL(n,{\mathbb{C}})}$ is simpler: each
$\pi\in\widehat{GL(n,{\mathbb{C}})}$ is a product of characters, where the
non-unitary characters come in pairs. As shown in \cite{Sahi-Kirillov,
Sahi-PAMS}, the adduced representation of $\pi$ is a product of the same form,
where each character of $G_{n}$ is restricted to $G_{n-1}$, and characters of
$G_{1}$ are thrown away. The associated partition of $\pi$ is determined in the
obvious way, similar to Theorem \ref{thm:AVTadic}.


\section{The $p$-adic case}
\label{sec:padic_short}

In this section we fix $F$ to be a non-Archimedean local field of characteristic zero and let $G_n:=GL(n,F)$.
%
%
%
%


\subsection{Definition of derivatives}

In the $p$-adic case, there is an additional definition of highest derivative,
using co-invariants. This definition works for all smooth admissible
representations of $G_{n}$, not only unitarizable. Moreover, Bernstein and
Zelevinsky define in \cite[\S 3]{BZ-Induced} all derivatives and not only the
highest ones in the following way.

Recall that $P_{n} = G_{n-1} \ltimes F^{n-1}$. Fix a non-trivial additive
character $\chi$ of $F$ and define a character $\theta_{n}$ of $F^{n-1}$ by
applying $\chi$ to the last coordinate. Denote by $\nu$ the determinant character $\nu(g)=|\det(g)|$.

Define two normalized coinvariants functors $\Psi^{-}\colon Rep(P_{n}) \to Rep(G_{n-1})$ and $\Phi^{-}\colon Rep(P_{n}) \to Rep(P_{n-1})$ by
$$
 \Psi^{-}(\tau) :=
\nu^{-1/2}\tau_{F^{n-1},1} \text{ and } \Phi^{-}(\tau) :=\nu^{-1/2} \tau_{F^{n-1},\theta_{n}}.
$$
Both functors are exact.

For a smooth representation $\tau$ of $P_{n}$ they define $\tau^{(k)}%
:=\Psi^{-}(\Phi^{-})^{k-1}\tau$ and call it the $k$-th derivative of $\tau$.
For a smooth representation $\pi$ of $G_{n}$, \emph{$k$-th derivative} is defined by
$D^{k}\pi:=\pi^{(k)}:=(\pi|_{P_{n}})^{(k)}$.

If $D^{k}\pi>0$ but $D^{k+l}\pi=0$ for any $l>0$ then $D^{k}\pi$ is called the
\emph{highest derivative} of $\pi$ and we denote it by $A(\pi)$, and $k$ is called the
\emph{depth} of $\pi$ and denoted $d(\pi)$.


For unitarizable representations, one can also define a \emph{shifted highest
derivative} of $\pi$ using Mackey theory, in the same way as adduced representation is defined in the Archimedean case (see
\S \S \ref{subsec:DerDepth}). By \cite{Ber}, this shifted highest derivative
will be isomorphic to $\nu^{1/2}A(\pi)$.


For a composition ${\alpha}=({\alpha}_{1},...,{\alpha}_{k})$ we define
$D^{{\alpha}}(\pi):=D^{{\alpha}_{1}}...D^{{\alpha}_{k}}(\pi).$ By
\cite[\S \S 8.3]{Zl} we have $(D^{{\alpha}_{k}}D^{{\alpha}_{k-1}}%
...D^{{\alpha}_{1}}\pi)^{*} \simeq Wh^{*}_{\alpha}(\pi)$.


Denote by $\mathcal{M}(G_{n})$ the category of smooth admissible representations of $G_n$ and define the Grothendieck ring $R = \bigoplus_{n} \Gamma(\mathcal{M}(G_{n}))$.
As an additive group it is the direct sum of Grothendieck groups of
$\mathcal{M}(G_{n})$ for all $n$, and the product is defined by parabolic
induction. Define the total derivative $D\colon R \to R$ as the sum of all
derivatives. By \cite[\S \S 4.5]{BZ-Induced}, $D$ is a homomorphism of rings.

\subsection{Zelevinsky classification}

In \cite{Zl}, Zelevinsky describes the generators of $R$ in the following way.
Denote by $C
:=\bigcup_{n} C_{n}$ the subset of all cuspidal irreducible representations of
$G_{n}$ for all $n$. For $\rho\in C_{d}$, a subset $\Delta\subset C_{d}$ of
the form $(\rho, \nu\rho, \nu^{2} \rho,...,\nu^{l-1}\rho)$ is called a
\emph{segment}. The representation $\nu^{(l-1)/2}\rho$ is called the center of
$\Delta$, the number $l$ is called the length of $\Delta$, and the number $d$
is called the depth of $\Delta$. We denote the set of all segments
$\Delta\subset C$ by $S$. We define a segment $\Delta^{-}$ by $\Delta^{-}
=(\rho, \nu\rho, \nu^{2} \rho,...,\nu^{l-2}\rho)$.

\begin{thm}
[\cite{Zl},\S 3 and \S \S 7.5]Let $\Delta=(\rho,\nu\rho,\nu^{2}\rho
,...,\nu^{l-1}\rho)\subset C_{d}$ be a segment. Then the representation
$\rho\times\nu\rho\times...\nu^{l-1}\rho$ contains a unique irreducible
constituent $\langle\Delta\rangle$ of the depth $d=depth(\Delta)$. Moreover,

\begin{enumerate}
\item $D(\langle\Delta\rangle)=\langle\Delta\rangle+\langle\Delta^{\prime
}\rangle$.

\item $R$ is a polynomial ring in indeterminates $\left\{  \langle
\Delta\rangle \colon \Delta\in S\right\}  $
\end{enumerate}
\end{thm}

Zelevinsky furthermore describes all irreducible representations in terms of
segments and shows, in \cite[\S \S 8.1]{Zl}, that $A$ maps irreducible
representations to irreducible.

\subsection{The wave front set}

Let $\pi\in\mathcal{M}(G_{n})$. Let $\chi_{\pi}$ be the character of $\pi$.
Then $\chi_{\pi}$ defines a distribution $\xi_{\pi}$ on a neighborhood of zero in ${\mathfrak{g}_{n}}$,
by restriction to a neighborhood of $1\in G$ and applying logarithm.
This distribution is known to be a combination of Fourier transforms of
Haar measures of nilpotent coadjoint orbits (\cite[p. 180]{HCWF}). This
enables to define a wave front cycle, $WFC(\pi)$ as a linear combination of orbits. Clearly, $WFC$
is additive on $R$. In \cite[II.1]{MW} it is shown that it is also
multiplicative, in the sense of Corollary \ref{cor:BarBoz}. The wave front set, $WF(\pi)$, is defined to be the set of orbits that appear in $WFC(\pi)$ with non-zero coefficients. Denote by $WFmax(\pi)$ the set of maximal elements of $WF(\pi)$ (with respect to the Bruhat ordering).
In \cite[II.2]%
{MW}, $WFC(\langle\Delta\rangle)$ is computed to be the orbit given by the
partition $\lambda_{\Delta}$ that has $length(\Delta)$ parts of size
$depth(\Delta)$, with multiplicity 1.
Thus, for an arbitrary smooth irreducible representation $\pi$ of $G_n$, $WFmax(\pi)$ consists of a single nilpotent orbit, given by the partition defined by the Zelevinsky classification (see \cite{OS2}).

\subsection{Proof of Theorem \ref{thm:UniformMain} over $F$}

Since both $D$ and in a sense $WFC$ are homomorphisms, and since segment
representations $\langle\Delta\rangle$ generate $R$, it is enough to prove
Theorem \ref{thm:UniformMain} for them. The above information implies that
$Wh_{\alpha}(\langle\Delta\rangle) = D^{{\alpha}}(\langle
\Delta\rangle) = 0$ unless ${\alpha} = \lambda_{\Delta}$. Now, $Wh_{\lambda
_{\Delta}} (\langle\Delta\rangle) = D^{\lambda_{\Delta}}(\Delta)=\langle
\emptyset\rangle= {\mathbb{C}}$. Since $WFC(\langle\Delta\rangle) =
{\mathcal{O}}_{\lambda_{\Delta}}$, this completes the proof of the main three parts
of Theorem \ref{thm:UniformMain}.

We can prove the ``moreover" part through an analog of Corollary
\ref{cor:SmallRep}: there is a Tadic classification of the unitary
representations as products of certain building blocks (see \cite{Tad}), from
which we see that $rank(\pi) =k< n/2$ if and only if $\pi= \chi\times\tau$,
for some character $\chi$ of $G_{n-k}$ and some $\tau\in\mathcal{M}(G_{k})$.
Since the same holds for Howe rank (by \cite[Part II, Corollary 3.2]{Sca}),
the ``moreover" part follows. One can also prove the ``moreover" part directly,
as was done in \cite[II.3]{MW} for the symplectic group.

\end{document}